\pdfoutput=1
\RequirePackage{ifpdf}
\ifpdf 
\documentclass[pdftex]{sigma}
\else
\documentclass{sigma}
\fi

\usepackage{booktabs}
\usepackage{longtable}
\usepackage{tikz-cd}
\usepackage{tikz}
\usetikzlibrary{arrows,decorations.pathmorphing,decorations.pathreplacing,positioning,shapes.geometric,shapes.misc,decorations.markings,decorations.fractals,calc,patterns}

\numberwithin{equation}{section}

\newtheorem{Theorem}{Theorem}[section]
\newtheorem*{Theorem*}{Theorem}
\newtheorem{Corollary}[Theorem]{Corollary}
\newtheorem{Lemma}[Theorem]{Lemma}
\newtheorem{Proposition}[Theorem]{Proposition}

\newtheorem{conje}[Theorem]{Conjecture}

 { \theoremstyle{definition}
\newtheorem{Definition}[Theorem]{Definition}

\newtheorem{Example}[Theorem]{Example}
\newtheorem{Remark}[Theorem]{Remark}

\newtheorem{conve}[Theorem]{Convention}
}

\newcommand{\NN }{\mathbb{N}}
\newcommand{\CC }{\mathbb{C}}

\newcommand{\ZZ }{\mathbb{Z}}
\newcommand{\id }{\mathrm{id}}
\DeclareMathOperator{\Aut}{Aut}

\DeclareMathOperator{\Hom}{Hom}

\newcommand{\Cc }{\mathcal{C}}
\newcommand{\Wg }{\mathcal{W}}
\newcommand{\rsC }{\mathcal{R}}
\newcommand{\re }{^\mathrm{re}}
\newcommand{\rer }[1]{(R\re)^{#1}}

\DeclareMathOperator{\Obj}{Obj}

\newcommand{\qq }{{\bf q}}
\newcommand{\KK }{\mathbb{K}}

\DeclareMathOperator{\Irr}{Irr}
\newcommand{\del }{\partial^L}
\newcommand{\der }{\partial^R}
\newcommand{\BB }{\mathfrak{B}}
\newcommand{\II }{\mathfrak{I}}
\newcommand{\Kc }{\mathcal{K}}
\DeclareMathOperator{\ad}{ad}
\newcommand{\dn }{d}
\newcommand{\qR }{\tilde R}
\newcommand{\bich }{\chi}

\usepackage[labelsep=period]{caption}

\begin{document}

\newcommand{\arXivNumber}{2204.05720}

\renewcommand{\PaperNumber}{019}

\FirstPageHeading

\ShortArticleName{Higher Braidings of Diagonal Type}

\ArticleName{Higher Braidings of Diagonal Type}

\Author{Michael CUNTZ and Tobias OHRMANN}

\AuthorNameForHeading{M.~Cuntz and T.~Ohrmann}

\Address{Leibniz Universit\"at Hannover,
Institut f\"ur Algebra, Zah\-lentheorie und Diskrete Mathematik,
Fakult\"at f\"ur Mathematik und Physik,
Wel\-fengarten 1,
D-30167 Hannover, Germany}
\Email{\href{mailto:cuntz@math.uni-hannover.de}{cuntz@math.uni-hannover.de}, \href{mailto:tobias@ohrmann.com}{tobias@ohrmann.com}}
\URLaddress{\url{https://www.iazd.uni-hannover.de/cuntz}}

\ArticleDates{Received May 30, 2022, in final form March 27, 2023; Published online April 06, 2023}

\Abstract{Heckenberger introduced the Weyl groupoid of a finite-dimensional Nichols algebra of diagonal type. We replace the matrix of its braiding by a~higher tensor and present a~construction which yields further Weyl groupoids. Abelian cohomology theory gives evidence for the existence of a higher braiding associated to such a tensor.}

\Keywords{Nichols algebra; braiding; Weyl groupoid}

\Classification{17B22; 16T30; 20F55}

\section{Introduction}

To a braided vector space $(V,c)$, i.e., a vector space $V$ and a linear isomorphism
\[
c \colon \ V \otimes V \rightarrow V \otimes V \qquad\text{with} \quad (c\otimes\id)(\id\otimes c)(c\otimes\id)=(\id\otimes c)(c\otimes\id)(\id\otimes c)
\]
one can associate a Hopf algebra $\BB(V) = T(V)/\II(V)$, where the ideal $\II(V)$ is the kernel of the symmetrizer of the braiding. This algebra $\BB(V)$ is called the Nichols algebra of $(V, c)$. For example, given a generalized Cartan matrix $A$ of finite type and a parameter $q\in\CC^\times$, it is well-known that the positive part of the Drinfeld--Jimbo quantum enveloping algebra
$U_q(A)$ is a Nichols algebra. This is a special case of a Nichols algebra of diagonal type.

The classification of finite-dimensional Nichols algebras of diagonal type \cite{p-H-09} was achieved using certain root systems. In classical Lie theory, these are orbits under the action of a Weyl group. For Nichols algebras, Heckenberger \cite{p-H-06} noticed that one has to consider Weyl groupoids, i.e., categories whose morphisms are compositions of reflections acting on a lattice.

Unfortunately, associating a Weyl groupoid to a Nichols algebra appears to be neither a~surjective nor an injective map (up to isomorphisms). At least in the case when the braiding defining the Nichols algebra is of diagonal type, one does not obtain all finite Weyl groupoids as classified in \cite{p-CH10}. To realize a much larger class of Weyl groupoids, Cuntz and Lentner \cite{p-CL-17} constructed Nichols algebras associated to restrictions of crystallographic arrangements. However, these are almost never Nichols algebras of Yetter--Drinfeld modules of finite groups. Moreover, we do not expect to obtain too many further Weyl groupoids of rank two in this way: in rank two, there are infinitely many finite Weyl groupoids, but only a very small finite set is associated to Nichols algebras of diagonal type. Restrictions of Weyl arrangements won't produce enough algebras using the construction by Cuntz and Lentner.

A braiding $c$ of diagonal type is completely determined by a matrix $\qq=(q_{ij})$ (when $c(x_i\otimes x_j) = q_{ij} x_j\otimes x_i$ for a basis $x_1,\dots,x_n$). To understand the action of a reflection $\sigma$ on a Nichols algebra, it is convenient to consider the bicharacter $\chi$ associated to $\qq$. The map $\sigma$ acts on the exponent vectors in $\ZZ^n$ of the generators of the Nichols algebra. But to obtain the new Nichols algebra $\sigma(\BB(V,c))$, one has to apply $\sigma$ on $\chi$. In an additive notation, this turns out to be the action of $\sigma\otimes\sigma$ on $\ZZ^n\otimes \ZZ^n$, see Sections~\ref{add_not} and~\ref{add_not_tensor}.

Conversely, given a Weyl groupoid $\Wg$, the tensor squares of its reflections define the variety of braidings $c$ such that $\Wg$ is the Weyl groupoid of $\BB(V,c)$. In additive notation it is very natural to generalize this machinery to higher tensor powers of reflections.
For example, fourth tensor powers would produce a variety of ``tensors'' $\qq=(q_{i,j,k,l})$.

To compute this variety, we need a rule to determine the reflections associated to such a~tensor. In the classical theory of Nichols algebras of diagonal type, this rule is given by a formula that we call \emph{Rosso's condition}:
the Cartan entry $c_{\ell,j}$ is minus the smallest $m\in\NN_0$ such that
\begin{equation*}
\big(1-q_{\ell,\ell}^{m} q_{\ell,j}q_{j,\ell}\big) \sum_{\nu=0}^{m} q_{\ell,\ell}^\nu = 0.
\end{equation*}
In Section \ref{generalized_Rosso}, we compute in a natural way a generalized Rosso condition for $d$-th tensor powers when $d$ is even.
In general, the Cartan entry $c_{\ell,j}$ is minus the smallest $m$ such that
\begin{equation*}
\left(1-\prod_{\nu=1}^{\dn} q_\nu^{\frac{1}{2}((m+1)^{\nu}-m^{\nu}+(-1)^{\nu+1})}\right)
\sum_{\mu=0}^m \prod_{\nu=1}^d q_{\dn-\nu}^{\mu \frac{1}{2}\left( (m+1)^{\nu-1}+\sum_{k=0}^{\nu-1} (-1)^{\nu-k} m^k \right)}
= 0,
\end{equation*}
where
\[
q_\nu :=\prod_{\substack{i_1,\dots,i_\dn\in\{\ell,j\}\\ |\{k \mid i_k=j\}|=\nu}} q_{i_1,\dots,i_\dn}.
\]
We also obtain a recursion for this condition (Theorem \ref{rosso_rec_d}) which is similar to the formula obtained in the classical case for the Nichols algebra.
We prove that our generalized Rosso condition produces Cartan matrices such that all axioms of a Weyl groupoid are satisfied (Theo\-rem~\ref{wg_from_q}). In examples, we obtain finite Weyl groupoids which do not occur as invariants of Nichols algebras of diagonal type. We conjecture that all finite Weyl groupoids are attained asymptotically when the tensor power (i.e.,~$d$) increases.

In Section \ref{ab_coh}, we give a precise definition of higher braidings in terms of abelian cohomology theory.
For an abelian group $G$, it is well-known that monoidal structures on the category of finite-dimensional $G$-graded vector spaces are (up to braided monoidal equivalence) classified by the third abelian cohomology $H_{ab}^3(G,\KK^\times)$.
If the associator is trivial, the defining abelian $3$-cocycle will simply be a~bicharacter on $G$. In this situation every object naturally becomes a~diagonally braided vector space.
In particular, if we set $G=\ZZ^n$, a braiding on $V=\oplus_{i=1}^n \KK \cdot x_i$ is uniquely determined by a matrix $\qq=(q_{ij})\in (\KK^\times)^{n \times n}$.
Two such braidings are equivalent if the numbers $q_{ii}$ and $q_{ij}q_{ji}$ coincide.
Consistently, we define a $d$-dimensional braiding to be an abelian $(2d-1)$-cocycle. Again, if the (higher) associative structure is trivial this will simply be a $d$-character on $G$. For $G=\ZZ^n$ the $d$-dimensional braiding is then
uniquely determined by a tensor $\qq=(q_{i_1,\dots,i_d})$. In analogy to the classical case we can associate the numbers $q_\nu \in \KK^\times$ (see above) to such a $d$-dimensional diagonal braiding and see that these numbers only depend on the cohomology class of the abelian $(2d-1)$-cocycle.

We also notice that, as in the classical case ($d=2$), the higher \textit{Rosso condition} as defined before only depends on the numbers $q_\nu \in \KK^\times$ and hence only on the cohomology class of the $d$-dimensional braiding $\qq$ (see Example \ref{ex: Rosso condition depends on Cohclass}). Moreover, we conjecture that the numbers $q_\nu \in \KK^\times$ are defined by multi-dimensional symmetric forms $\theta_\lambda \in H^{2d-1}(G,\KK^\times)$ indexed by partitions $\lambda$ of $d$. This is Conjecture \ref{conje: symmetrized cycles are n-forms in arguments} and Remark \ref{rem: n-form}. For the case $d=2$, we obtain the well-known quadratic form $\theta_2$ and symmetric bilinear form $\theta_{(1,1)}$, which define the numbers $q_{ii}$ and $q_{ij}q_{ji}$ in the diagonal case.

The computations in Section \ref{ab_coh} require explicit knowledge of low-dimensional generators and boundaries of the abelian chain complex. In Table \ref{tab:generators of abelian chain complex}, we list these up to $n=6$, as we believe this can be helpful to other people dealing with abelian chain complexes independently from our motivation.

\section{Nichols algebras of diagonal type}

\subsection{Nichols algebras}\label{sec: Nichols algebras}

Let $G$ be an abelian group and $V$ a finite-dimensional representation of the Drinfeld double~$D(G)$ over the field~$\KK$. Since $G$ is abelian, $V$ decomposes into a direct sum of one-dimensional modules%
\[ V = \bigoplus_{i=1}^n V_{g_i,\xi_i}, \]
where $g_1,\dots,g_n\in G$, $\xi_1,\dots,\xi_n\in \hat G=\Irr(G)$. Choose $x_1,\dots,x_n\in V$ such that $V_{g_i,\xi_i}=\langle x_i\rangle$ for all $i=1,\dots,n$.
Then $G$ acts on $V$ via
\[
g\cdot x_i = \xi_i(g) x_i, \qquad g\in G,
\]
and $\hat G$ acts on $V$ via
\[
\xi\cdot x_i = \xi(g_i) x_i, \qquad \xi\in \hat G.
\]
In this section, the relevant information about $V$ is encoded in the matrix
\[ \qq = (q_{i,j})_{i,j} := (\xi_j(g_i))_{i,j}. \]

There are many equivalent definitions of a Nichols algebra. For our purpose, the definition via skew derivations will be the most convenient (cf.~\cite{p-H-06}).

\begin{Definition}
Define (left and right) {\it skew derivations}
$\del_i \colon V^{\otimes k}\rightarrow V^{\otimes k-1}$,
$\der_i \colon V^{\otimes k}\rightarrow V^{\otimes k-1}$,
$i=1,\dots,n$, $k\in\NN_0$ inductively by
\begin{alignat*}{4}
& \del_i(1):=0, \qquad &&\del_i(x_j):=\delta_{i,j}, \qquad&&
\del_i(xy):=\del_i(x)y + (\xi_i\cdot x) \del_i(y),&\\
& \der_i(1):=0, \qquad &&\der_i(x_j):=\delta_{i,j}, \qquad&&
\der_i(xy):=x\der_i(y) + \der_i(x) (g_i \cdot y),&
\end{alignat*}
for $x$, $y$ of degree $\ge 1$.
\end{Definition}

Let $\varepsilon \colon T(V) \rightarrow \KK$ be the homomorphism of algebras with $\varepsilon(v)=0$ for all $v\in V$.

\begin{Definition}
Let $\II(V)$ be the largest ideal among all ideals $I$ of $T(V)$ such that $\varepsilon(I)=0$ and
$\del_i(I)\subset I$ for all $i=1,\dots,n$.
Then $\BB(V):=T(V)/\II(V)$ is called the {\it Nichols algebra $($of diagonal type$)$} of~$V$.
\end{Definition}

Using methods similar to \cite[Lemma~2.2]{p-J-77}, one can prove the following result:

\begin{Proposition}
Let $i\in\{1,\dots,n\}$, $\Kc_i:=\ker \der_i$ and
\[ h_i := \min\left\{ m\ge 1 \mid \sum_{k=0}^{m-1} q_{i,i}^k =0\right\} \cup \{\infty\}. \]
Then
\[ \BB(V) \cong \begin{cases}
\Kc_i\otimes \KK[x_i], & h_i = \infty, \\
\Kc_i\otimes \KK[x_i]/\big(x_i^{h_i}\big), & h_i \in \NN.
\end{cases} \]
as $\ZZ^n$-graded vector spaces over $\KK$.
\end{Proposition}

\subsection{Adjoint action}

Following \cite{p-H-06}, we will construct a (generalized) root system associated to a Nichols algebra of diagonal type.
An essential tool is the adjoint action:

\begin{Definition}With the above notation, for $i=1,\dots,n$ and $y\in T(V)$, let
\[ (\ad x_i)(y) := x_i y - (g_i\cdot y) x_i. \]
\end{Definition}

In \cite[Proposition~1]{p-H-06}, it is shown that the algebra $\Kc_i$ is the subalgebra of $\BB(V)$ generated by the elements $(\ad x_i)^m(x_j)$, $m\ge 0$, $j\ne i$. This implies:

\begin{Proposition}\label{Ki_fg}
Let $i\in\{1,\dots,n\}$. Then $\Kc_i$ is finitely generated if and only if: for all $j\ne i$, there exists $m_{i,j}\ge 0$ such that
\[ (\ad x_i)^{m_{i,j}+1}(x_j) \in \II(V). \]
\end{Proposition}

The following proposition will be the key for our generalization, hence we sketch a proof:

\begin{Proposition}
In the situation of Proposition~{\rm \ref{Ki_fg}}, we have
\begin{equation}\label{rosso}
m_{i,j} = \min \left\{ m\ge 0\mid \sum_{k=0}^m q_{i,i}^k = 0\ \text{or} \ q_{i,i}^{m} q_{i,j}q_{j,i} = 1 \right\}.
\end{equation}
\end{Proposition}
\begin{proof}
Let $y\in T(V)$. Since $g_i \del_i(y) = q_{i,i}^{-1} \del_i(g_i y)$,
\begin{align*}
\del_i((\ad x_i)(y)) &= y+q_{i,i} x_i \del_i(y) - \del_i(g_i y) x_i - (\xi_i g_i y) \\
&= y - (\xi_i g_i y) + q_{i,i} (\ad x_i)\del_i(y).
\end{align*}
In particular, for $y=(\ad x_i)^{k}(x_j)$, $k\ge 0$ we get
\[ \del_i\big((\ad x_i)^{k+1}(x_j)\big) =
\big(1-q_{i,i}^{2k}q_{i,j}q_{j,i}\big) (\ad x_i)^{k}(x_j)
+ q_{i,i} (\ad x_i) \del_i\big((\ad x_i)^{k}(x_j)\big). \]
With
\begin{equation}\label{rosso_rec}
R_0:=1-q_{i,j}q_{j,i}, \qquad
R_k:=
1-q_{i,i}^{2k} q_{i,j}q_{j,i} + q_{i,i} R_{k-1}
\end{equation}
for $k>0$, we obtain
\[ \del_i\big((\ad x_i)^{k}(x_j)\big) = R_{k-1} (\ad x_i)^{k-1}(x_j). \]
Using induction we see that
\begin{equation}\label{Rrosso}
R_k = \big(1-q_{i,i}^{k} q_{i,j}q_{j,i}\big) \sum_{\nu=0}^{k} q_{i,i}^\nu.
\end{equation}
Thus $(\ad x_i)^{k}(x_j)\in\II(V)$ if and only if $R_{k-1}=0$.
\end{proof}

Let us call \eqref{rosso} and \eqref{Rrosso} the {\it Rosso condition} because it goes back to Rosso~\cite{MR1632802}.

\subsection{The Weyl groupoid of a Nichols algebra}

Recall that the Nichols algebra of diagonal type is completely determined by the matrix $\qq$.

\begin{Definition}
The {\it Cartan matrix} $\big(c^\qq_{i,j}\big)_{1\le i,j\le n}$ of $\qq$ is given by the formulas $c^\qq_{i,i}=2$ and~\eqref{rosso}:
\begin{gather}\label{rosso_cartan}
c^\qq_{i,j} = -\min\big\{ m\in \NN_0 \mid 1+q_{ii}+q_{ii}^2+\cdots+q_{ii}^m=0 \ \text{or}\ q_{ii}^m q_{ij}q_{ji}=1 \big\}, \qquad i\ne j.
\end{gather}
\end{Definition}
Notice that depending on $\qq$, some of these entries may not be defined since an $m\in\NN_0$ satisfying one of the two conditions possibly does not exist. In this case, we will say that the Nichols algebra has no corresponding Weyl groupoid.
Proposition~\ref{Ki_fg} ensures that a Cartan matrix exists in the special case that the Nichols algebra is finite dimensional.

\begin{Definition}[{\cite{p-H-06}}]
Let $\ZZ^n = \langle \alpha_1,\dots,\alpha_n\rangle$ and $\alpha_1,\dots,\alpha_n$ be the standard basis.
Let $\ell\in\{1,\dots,n\}$ and $\sigma_\ell\in\Aut(\ZZ^n)$ be the reflection given by
\[ \sigma_\ell(\alpha_j):=\alpha_j - c^\qq_{\ell,j}\alpha_\ell \]
for $\ell=1,\dots,n$. Using $\sigma_\ell$ we obtain a new matrix $\sigma_\ell(\qq)$ via
\begin{gather}\label{reflect_q}
\sigma_\ell(\qq)_{i,j} := \prod_{k,s=1}^n q_{k,s}^{\sigma_\ell(\alpha_i)_k \sigma_\ell(\alpha_j)_s}.
\end{gather}
We can compute a Cartan matrix to $\sigma_\ell(\qq)$ instead of $\qq$ using Rosso's formula and apply another reflection and so on.
This way we obtain a groupoid called the {\it Weyl groupoid} whose objects are the matrix $\qq$ and its images and the morphisms are compositions of the reflections between these matrices.
We recall the precise definition of a Weyl groupoid in the following subsection.

The {\it Dynkin diagram} of $\qq$ is a graph with vertices $1,\dots,n$ and edges $(i,j)$ when $q_{ij}q_{ji}\ne 1$ and labeled by $q_{ij}q_{ji}$.
\end{Definition}

\begin{Example}\label{wg_nichols_example}
Let $n=3$, $\zeta$ be a primitive third root of unity, and
\[ \qq=(q_{ij})_{i,j}=\begin{pmatrix} -1 & \hphantom{-}\zeta & \hphantom{-}\zeta \\ \hphantom{-}1 & -1 & \hphantom{-}\zeta \\\hphantom{-}1 & \hphantom{-}1 & -1 \end{pmatrix}. \]

\begin{figure}[t]\centering
 \makeatletter%
 \providecommand\color[2][]{%
 \errmessage{(Inkscape) Color is used for the text in Inkscape, but the package 'color.sty' is not loaded}%
 \renewcommand\color[2][]{}%
 }%
 \providecommand\transparent[1]{%
 \errmessage{(Inkscape) Transparency is used (non-zero) for the text in Inkscape, but the package 'transparent.sty' is not loaded}%
 \renewcommand\transparent[1]{}%
 }%
 \providecommand\rotatebox[2]{#2}%
 \newcommand*\fsize{\dimexpr\f@size pt\relax}%
 \newcommand*\lineheight[1]{\fontsize{\fsize}{#1\fsize}\selectfont}%
 \ifx\svgwidth\undefined%
 \setlength{\unitlength}{286.85038468bp}%
 \ifx\svgscale\undefined%
 \relax%
 \else%
 \setlength{\unitlength}{\unitlength * \real{\svgscale}}%
 \fi%
 \else%
 \setlength{\unitlength}{\svgwidth}%
 \fi%
 \global\let\svgwidth\undefined%
 \global\let\svgscale\undefined%
 \makeatother%
 \begin{picture}(1,0.86704021)%
 \lineheight{1}%
 \setlength\tabcolsep{0pt}%
 \put(0,0){\includegraphics[width=\unitlength,page=1]{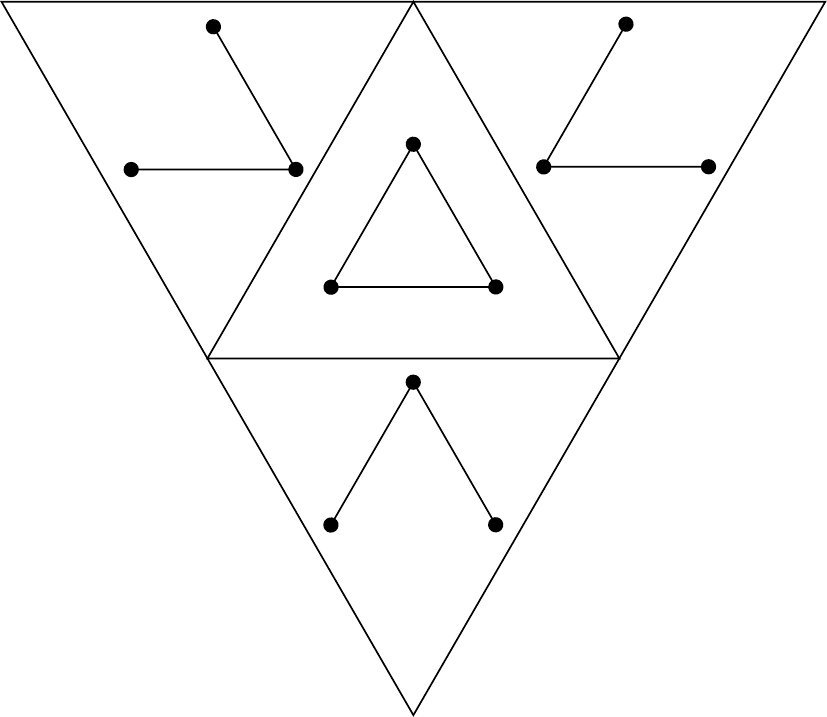}}%
 \put(0.5615138,0.61816383){\color[rgb]{0,0,0}\makebox(0,0)[lt]{\lineheight{1.25}\smash{\begin{tabular}[t]{l}$\zeta$\end{tabular}}}}%
 \put(0.41816848,0.61716841){\color[rgb]{0,0,0}\makebox(0,0)[lt]{\lineheight{1.25}\smash{\begin{tabular}[t]{l}$\zeta$\end{tabular}}}}%
 \put(0.48884572,0.4887548){\color[rgb]{0,0,0}\makebox(0,0)[lt]{\lineheight{1.25}\smash{\begin{tabular}[t]{l}$\zeta$\end{tabular}}}}%
 \put(0.20514137,0.82621371){\color[rgb]{0,0,0}\makebox(0,0)[lt]{\lineheight{1.25}\smash{\begin{tabular}[t]{l}$\zeta$\end{tabular}}}}%
 \put(0.13446415,0.68685018){\color[rgb]{0,0,0}\makebox(0,0)[lt]{\lineheight{1.25}\smash{\begin{tabular}[t]{l}$\zeta$\end{tabular}}}}%
 \put(0.7795182,0.83019549){\color[rgb]{0,0,0}\makebox(0,0)[lt]{\lineheight{1.25}\smash{\begin{tabular}[t]{l}$\zeta$\end{tabular}}}}%
 \put(0.83725451,0.69381832){\color[rgb]{0,0,0}\makebox(0,0)[lt]{\lineheight{1.25}\smash{\begin{tabular}[t]{l}$\zeta$\end{tabular}}}}%
 \put(0.41617759,0.19708696){\color[rgb]{0,0,0}\makebox(0,0)[lt]{\lineheight{1.25}\smash{\begin{tabular}[t]{l}$\zeta$\end{tabular}}}}%
 \put(0.56748658,0.19708685){\color[rgb]{0,0,0}\makebox(0,0)[lt]{\lineheight{1.25}\smash{\begin{tabular}[t]{l}$\zeta$\end{tabular}}}}%
 \put(0.48486394,0.71571828){\color[rgb]{0,0,0}\makebox(0,0)[lt]{\lineheight{1.25}\smash{\begin{tabular}[t]{l}$-1$\end{tabular}}}}%
 \put(0.61228203,0.49472752){\color[rgb]{0,0,0}\makebox(0,0)[lt]{\lineheight{1.25}\smash{\begin{tabular}[t]{l}$-1$\end{tabular}}}}%
 \put(0.36839584,0.49074574){\color[rgb]{0,0,0}\makebox(0,0)[lt]{\lineheight{1.25}\smash{\begin{tabular}[t]{l}$-1$\end{tabular}}}}%
 \put(0.66504109,0.63210011){\color[rgb]{0,0,0}\makebox(0,0)[lt]{\lineheight{1.25}\smash{\begin{tabular}[t]{l}$-1$\end{tabular}}}}%
 \put(0.52070035,0.40115485){\color[rgb]{0,0,0}\makebox(0,0)[lt]{\lineheight{1.25}\smash{\begin{tabular}[t]{l}$-1$\end{tabular}}}}%
 \put(0.3504778,0.69680457){\color[rgb]{0,0,0}\makebox(0,0)[lt]{\lineheight{1.25}\smash{\begin{tabular}[t]{l}$-1$\end{tabular}}}}%
 \put(0.55554107,0.32550046){\color[rgb]{0,0,0}\makebox(0,0)[lt]{\lineheight{1.25}\smash{\begin{tabular}[t]{l}$\zeta^{-1}$\end{tabular}}}}%
 \put(0.41020487,0.3274913){\color[rgb]{0,0,0}\makebox(0,0)[lt]{\lineheight{1.25}\smash{\begin{tabular}[t]{l}$\zeta^{-1}$\end{tabular}}}}%
 \put(0.66504109,0.76748198){\color[rgb]{0,0,0}\makebox(0,0)[lt]{\lineheight{1.25}\smash{\begin{tabular}[t]{l}$\zeta^{-1}$\end{tabular}}}}%
 \put(0.74965457,0.62712291){\color[rgb]{0,0,0}\makebox(0,0)[lt]{\lineheight{1.25}\smash{\begin{tabular}[t]{l}$\zeta^{-1}$\end{tabular}}}}%
 \put(0.30568221,0.77445012){\color[rgb]{0,0,0}\makebox(0,0)[lt]{\lineheight{1.25}\smash{\begin{tabular}[t]{l}$\zeta^{-1}$\end{tabular}}}}%
 \put(0.23600044,0.62513208){\color[rgb]{0,0,0}\makebox(0,0)[lt]{\lineheight{1.25}\smash{\begin{tabular}[t]{l}$\zeta^{-1}$\end{tabular}}}}%
 \end{picture}%
\caption{Dynkin diagrams in the Weyl groupoid of Example \ref{wg_nichols_example}.}\label{fig_wgex}
\end{figure}
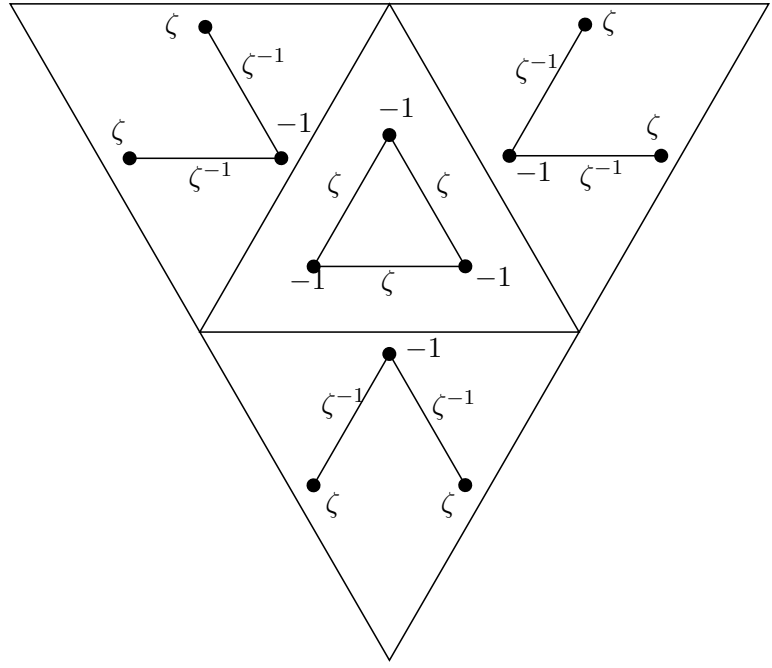

Figure~\ref{fig_wgex} displays the Weyl groupoid obtained from $\qq$: the Dynkin diagrams are the objects, reflections are indicated by separating lines.
\end{Example}

\subsection{Definition of a Weyl groupoid}

\subsubsection{Cartan graphs}

The collection of Cartan matrices produced in Example \ref{wg_nichols_example} is an example of what is known as a~Cartan graph.\footnote{In early papers these were called Cartan schemes.}
To define the corresponding Weyl groupoid, we need several notions (compare~\mbox{\cite{p-CH09b, p-CH09a}}).

\begin{Definition}[{\cite[Section~1.1]{b-Kac90}}]
Let $I$ be a non-empty finite set and $\{\alpha_i \mid i\in I\}$ the standard basis of $\ZZ^I$.
A {\it generalized Cartan matrix} $C=(c_{ij})_{i,j \in I}$ is a matrix in $\ZZ^{I \times I}$ such that
\begin{enumerate}\setlength{\leftskip}{0.10cm}\itemsep=0pt
\item[(M1)] $c_{ii}=2$ and $c_{jk}\le 0$ for all $i,j,k\in I$ with $j\not=k$,
\item[(M2)] if $i,j\in I$ and $c_{ij}=0$, then $c_{ji}=0$.
\end{enumerate}
\end{Definition}

\begin{Definition}
Let $A$ be a non-empty set, $\rho_i\colon A \to A$ a map for all $i\in I$, and $C^a=\big(c^a_{jk}\big)_{j,k \in I}$ a generalized Cartan matrix in $\ZZ ^{I \times I}$ for all $a\in A$. The quadruple
\[ \Cc = \Cc \big(I,A,(\rho_i)_{i \in I}, \big(C ^a\big)_{a \in A}\big) \]
is called a {\it Cartan graph} if
\begin{enumerate}\itemsep=0pt
\item[(C1)] $\rho_i^2 = \id$ for all $i \in I$,
\item[(C2)] $c^a_{ij} = c^{\rho_i(a)}_{ij}$ for all $a\in A$ and $i,j\in I$.
\end{enumerate}
\end{Definition}

\begin{Definition}
Let $\Cc = \Cc \big(I,A,(\rho_i)_{i \in I}, \big(C ^a\big)_{a \in A}\big)$ be a Cartan graph.
For all $i \in I$ and $a \in A$ define $\sigma _i^a \in \Aut\big(\ZZ ^I\big)$ by
\begin{align*}
\sigma _i^a \big(\alpha_j\big) = \alpha_j - c_{ij}^a \alpha_i \qquad
\text{for all $j \in I$.}
\end{align*}
The {\it Weyl groupoid of} $\Cc $ is the category $\Wg (\Cc )$ such that $\Obj(\Wg (\Cc ))=A$ and
the morphisms are compositions of maps $\sigma _i^a$ with $i\in I$ and $a\in A$,
where $\sigma _i^a$ is considered as an element in $\Hom (a,\rho_i(a))$.
The category $\Wg (\Cc )$ is a groupoid in the sense that all morphisms are isomorphisms.
The cardinality of $I$ is the {\it rank of} $\Wg (\Cc )$.
\end{Definition}

Since most upper indices referring to elements of $A$ are determined by the context, we will often omit them to improve readability.

\subsubsection{Root systems}

\begin{Definition}
Let $\Cc $ be a Cartan graph. For all $a\in A$ let
\[ \rer a=\big\{ \id ^a \sigma _{i_1}\cdots \sigma_{i_k}(\alpha_j)\,|\,
k\in \NN _0,\,i_1,\dots,i_k,j\in I\big\}\subseteq \ZZ ^I.\]
The elements of the set $\rer a$ are called {\it real roots} (at $a$).
The pair $\big(\Cc ,\big(\rer a\big)_{a\in A}\big)$ is denoted by $\rsC \re (\Cc )$.
A real root $\alpha\in \rer a$, where $a\in A$, is called positive
(resp.\ negative) if $\alpha\in \NN _0^I$ (resp.\ $\alpha\in -\NN _0^I$).
\end{Definition}

The Weyl groupoids which are relevant for the study of Nichols algebras produce sets of real roots which satisfies additional properties:

\begin{Definition}
Let $\Cc =\Cc \big(I,A,(\rho_i)_{i\in I},\big(C ^a\big)_{a\in A}\big)$ be a Cartan graph.
For all $a\in A$ let $R^a\subseteq \ZZ ^I$, and define
$m_{i,j}^a= \big|R^a \cap \big(\NN_0 \alpha_i + \NN_0 \alpha_j\big)\big|$ for all $i,j\in
I$ and $a\in A$. We say that
\[ \rsC = \rsC \big(\Cc , \big(R^a\big)_{a\in A}\big) \]
is a {\it root system of type} $\Cc $, if it satisfies the following axioms.
\begin{enumerate}\itemsep=0pt
\item[(R1)]
$R^a=R^a_+\cup - R^a_+$, where $R^a_+=R^a\cap \NN_0^I$, for all $a\in A$.
\item[(R2)]
$R^a\cap \ZZ\alpha_i=\{\alpha_i,-\alpha_i\}$ for all $i\in I$, $a\in A$.
\item[(R3)]
$\sigma _i^a\big(R^a\big) = R^{\rho_i(a)}$ for all $i\in I$, $a\in A$.
\item[(R4)]
If $i,j\in I$ and $a\in A$ such that $i\not=j$ and $m_{i,j}^a$ is finite, then
$(\rho_i\rho_j)^{m_{i,j}^a}(a)=a$.
\end{enumerate}
\end{Definition}

The axioms (R2) and (R3) are always fulfilled for $\rsC \re $.
The root system $\rsC $ is called {\it finite} if for all $a\in A$ the set $R^a$ is finite.
By \cite[Proposition~2.12]{p-CH09a}, if $\rsC $ is a finite root system
of type $\Cc $, then $\rsC =\rsC \re $, and hence $\rsC \re $ is a root
system of type $\Cc $ in that case.

\section{Tensor powers of reflections}\label{sec: tensor powers}

\subsection{Additive notation}\label{add_not}
Let us take a closer look to Rosso's formula \eqref{rosso_cartan}. For $i\ne j$, we are looking for a smallest $m\in\NN_0$ such that
\begin{equation*}
1+q_{ii}+q_{ii}^2+\cdots+q_{ii}^m=0 \qquad\text{or}\qquad q_{ii}^m q_{ij}q_{ji}=1,
\end{equation*}
or equivalently,
\begin{equation}\label{rosso_classic2}
\big(q_{ii}^{m+1}=1\quad\text{and}\quad q_{ii}\ne 1\big) \qquad\text{or}\qquad q_{ii}^m q_{ij}q_{ji}=1.
\end{equation}
This shows that in the whole procedure computing the Weyl groupoid, only integral powers of~$q_{ij}$ are involved (remember \eqref{reflect_q}).
This allows us to write everything in an additive notation:
Let $W:=\ZZ^{n\times n} = \langle e_{ij}\mid i,j\in \{1,\dots,n\}\rangle$ and $\bich_\qq$ be the group homomorphism
\[ \bich_\qq \colon \ (W,+) \rightarrow U:=\langle q_{ij} \mid i,j\in \{1,\dots,n\}\rangle\le \KK^\times,
\qquad e_{ij}\mapsto q_{ij}. \]
Then we obtain an isomorphism of $\ZZ$-modules
\[ \varepsilon \colon \ U \rightarrow M_\qq:=W/\ker(\bich_\qq). \]
Using the map $\varepsilon$, Rosso's condition (\ref{rosso_classic2}) becomes
\begin{gather*}
\big((m+1)\varepsilon(q_{ii})=0\quad\text{and}\quad \varepsilon(q_{ii})\ne 0\big) \qquad\text{or}\qquad m\varepsilon(q_{ii})+\varepsilon(q_{ij})+\varepsilon(q_{ji})=0
\end{gather*}
or with $\overline{e}_{ij}:=e_{ij}+\ker(\bich_\qq)$ for $i,j=1,\dots,n$ in $M_\qq$,
\begin{gather*}
\big((m+1)\overline{e}_{ii}=0\quad\text{and}\quad \overline{e}_{ii}\ne 0\big) \qquad\text{or}\qquad m\overline{e}_{ii}+\overline{e}_{ij}+\overline{e}_{ji}=0.
\end{gather*}

\subsection{Tensor products and eigenvectors}\label{add_not_tensor}
The reflection of matrices defined in (\ref{reflect_q}),
\[ \sigma_\ell(\qq)_{i,j} := \prod_{k,s=1}^n q_{k,s}^{\sigma_\ell(\alpha_i)_k \sigma_\ell(\alpha_j)_s} \]
translates to a tensor product in additive notation:
\[
\begin{tikzcd}
\ZZ^n \arrow[r, "\otimes"] & W\cong \ZZ^n\otimes\ZZ^n \arrow[dr, "\bich_{\sigma_\ell(\qq)}"]
& \\
 & & U \le \KK^\times ,\\
 \ZZ^n \arrow[r, "\otimes"] \arrow[uu, "\sigma_\ell"]
 & W\cong \ZZ^n\otimes\ZZ^n \arrow[ur, "\bich_{\qq}", swap] \arrow[uu, "\sigma_\ell \otimes \sigma_\ell"] &
\end{tikzcd}
\]
where $\ZZ^n \stackrel{\otimes}{\longrightarrow} W$ is the map $v \mapsto v\otimes v$.

To obtain a Weyl groupoid, it is necessary that Rosso's formula satisfies the symmetry
\begin{gather*}
c^\qq_{\ell,j} = c^{\sigma_\ell(\qq)}_{\ell,j}
\end{gather*}
for all $\ell,j\in \{1,\dots,r\}$. This comes from the fact that
$(m+1) e_{\ell \ell}$ and $m e_{\ell \ell}+e_{\ell j}+e_{j \ell}$ are eigenvectors of $\sigma_\ell \otimes \sigma_\ell$ to the eigenvalues $\pm 1$ for $m = -c^\qq_{\ell,j}$.
Because of this fact, the conditions will remain the same after reflecting.
In order to ``explain'' this fact, we now compare the degrees of $(\ad x_\ell)^{m+1}(x_j)$ and $(\ad x_\ell)^{m}(x_j)$:
\[ u_m := ((m+1)\alpha_\ell+\alpha_j)^{\otimes 2} - (m\alpha_\ell+\alpha_j)^{\otimes 2} = (2m+1)\alpha_\ell\otimes\alpha_\ell + \alpha_\ell\otimes \alpha_j + \alpha_j\otimes \alpha_\ell. \]
The vector $u_m$ is a sum of eigenvectors of $\sigma_\ell\otimes\sigma_\ell$:
\[ u_m = \underbrace{\frac{1}{2}\big(\sigma_\ell^{\otimes 2}(u_m)+u_m\big)}_{v_m}
\underbrace{-\frac{1}{2}\big(\sigma_\ell^{\otimes 2}(u_m)-u_m\big)}_{w_m}
, \]
these are
\begin{gather*}
v_m := \frac{1}{2}\big(\sigma_\ell^{\otimes 2}(u_m)+u_m\big)
 = (m+1)\alpha_\ell^{\otimes 2}
, \\
w_m := -\frac{1}{2}\big(\sigma_\ell^{\otimes 2}(u_m)-u_m\big)
 = m\alpha_\ell^{\otimes 2}+\alpha_\ell\otimes \alpha_j + \alpha_j\otimes \alpha_\ell.
\end{gather*}
We observe that Rosso's condition \eqref{Rrosso} becomes
\begin{equation}\label{R_add}
(1-q_{\ell,\ell})R_{m} = (1-\bich_\qq(w_m))(1-\bich_\qq(v_m)).
\end{equation}

\subsection{Higher tensor powers}\label{generalized_Rosso}
We now generalize the additive decomposition into eigenvectors \eqref{R_add} to higher tensor powers.
Let $\dn\in\NN$ and
\[\qq = (q_{i_1,\dots,i_\dn})_{i_1,\dots,i_\dn}\in \KK^{\{1,\dots,n\}^\dn}.
\]
Let $W:=\big(\ZZ^n\big)^{\otimes \dn}$ and denote by $e_{i_1,\dots,i_\dn}:=\alpha_{i_1}\otimes\cdots\otimes\alpha_{i_\dn}$ the elements of its standard basis.
Let~$\bich_\qq$ be the group homomorphism
\[ \bich_\qq \colon \ (W,+) \rightarrow U:=\big\langle q_{i_1,\dots,i_\dn} \mid i_1,\dots,i_\dn\in \{1,\dots,n\}\big\rangle\le \KK^\times,
\qquad e_{i_1,\dots,i_\dn}\mapsto q_{i_1,\dots,i_\dn}. \]
As before, we compare the $\dn$-th tensor powers of presumed degrees:
\[ u_m := ((m+1)\alpha_\ell+\alpha_j)^{\otimes \dn} - (m\alpha_\ell+\alpha_j)^{\otimes \dn}. \]
We decompose $u_m$ as a sum of eigenvectors of $\sigma_\ell^{\otimes \dn}$:
\[ u_m = \underbrace{\frac{1}{2}\big(\sigma_\ell^{\otimes \dn}(u_m)+u_m\big)}_{v_m}
\underbrace{-\frac{1}{2}\big(\sigma_\ell^{\otimes \dn}(u_m)-u_m\big)}_{w_m}. \]
We do not know yet what will be the analog for $d>2$ of the factor $(1-q_{\ell,\ell})$, so let us first set
\begin{equation*}
\qR_m := \big(1-\bich_\qq(v_m)\big)\big(1-\bich_\qq(w_m)\big).
\end{equation*}
In the following Theorem \ref{rosso_rec_d}, we find a recursion for $\qR_m$ similar to equation \eqref{rosso_rec}. We need some notation for its formulation:

\begin{Definition}\label{def: defn of qk's}
Let $\ell,j\in\{1,\dots,n\}$.
For $k\in\{0,\dots,\dn\}$ and $m\in\NN_0$ we set
\begin{gather*}
\gamma_k := \sum_{\substack{i_1,\dots,i_\dn\in\{\ell,j\}\\ |\{\nu\:\mid\: i_\nu=j\}|=k}} \alpha_{i_1}\otimes \dots \otimes \alpha_{i_\dn},
\qquad
 q_k :=\chi_\qq(\gamma_k)= \prod_{\substack{i_1,\dots,i_\dn\in\{\ell,j\}\\ |\{\nu\:\mid\: i_\nu=j\}|=k}} q_{i_1,\dots,i_\dn},\\
 e_{m,k}:=\sum_{\substack{\nu=0 \\ (-1)^\nu=(-1)^k}}^{k-2} \binom{k}{\nu} m^\nu,
\qquad
r_m:=\prod_{i=0}^{\dn-2} q_i^{e_{m,\dn-i}},\\
f_{m,k} := \sum_{\substack{\nu=0 \\ (-1)^\nu=(-1)^{k+1}}}^{k-1} \binom{k}{\nu} m^\nu,
\qquad
z_m := \prod_{i=0}^{\dn-1} q_i^{f_{m,\dn-i}}.
\end{gather*}
\end{Definition}

\begin{Theorem}\label{rosso_rec_d}
With the above notation, $\qR_m$, $m\in\NN$ satisfy the following recursions:
\begin{gather}
\label{rec0} \qR_0 = (1-r_0)(1-z_0),\\
\label{rec1} \qR_m = r_m \qR_{m-1} + (1-r_m)(1-z_m).
\end{gather}
\end{Theorem}

\begin{proof}
For equation \eqref{rec0}, we compute
\begin{align*}
v_0 &= \frac{1}{2} \big(\sigma_\ell^{\otimes \dn}(u_0)+u_0\big)
= \frac{1}{2}\big((-\alpha_\ell+\alpha_j)^{\otimes \dn} - \alpha_j^{\otimes \dn}
 + (\alpha_\ell+\alpha_j)^{\otimes \dn} - \alpha_j^{\otimes \dn} \big) \\
&= \frac{1}{2}\sum_{\nu=0}^{d-1}\big(1+(-1)^{\dn-\nu}\big)\gamma_\nu
= \sum_{\nu=0}^{d-2} e_{0,\dn-\nu}\gamma_\nu
\end{align*}
because $e_{0,k}=\frac{1}{2}\big(1+(-1)^k\big)$; hence $\chi_\qq(v_0)=r_0$.
We check in the same way that $\chi_\qq(w_0)=z_0$, thus $\qR_0=(1-r_0)(1-z_0)$.
For the recursions, observe first that using
\[e_{m,k}+f_{m,k} = (m+1)^k-m^k\] we obtain
\begin{equation}\label{efef}
e_{m,k}-f_{m,k} = (m-1)^k-m^k = -e_{m-1,k}-f_{m-1,k}.
\end{equation}
This implies the relations
\begin{gather}
\label{ef1} \prod_{\nu=0}^{d-1} q_\nu^{f_{m,\dn-\nu}} \stackrel{\eqref{efef}}{=}
\prod_{\nu=0}^{d-2} q_\nu^{e_{m,\dn-\nu}}
\prod_{\nu=0}^{\dn-1} q_\nu^{e_{m-1,\dn-\nu}+f_{m-1,\dn-\nu}}, \\
\label{ef2} \prod_{\nu=0}^{d-2} q_\nu^{e_{m,\dn-\nu}}
\prod_{\nu=0}^{\dn-1} q_\nu^{\frac{1}{2}(e_{m-1,\dn-\nu}+f_{m-1,\dn-\nu})}
\stackrel{\eqref{efef}}{=} \prod_{\nu=0}^{\dn-1} q_\nu^{\frac{1}{2}(e_{m,\dn-\nu}+f_{m,\dn-\nu})}.
\end{gather}
We now compute $v_m$ and $w_m$:
\begin{align*}
v_m &= \frac{1}{2} \big(\sigma_\ell^{\otimes \dn}(u_m)+u_m\big) \\
&= \frac{1}{2}\big(\big(-\alpha_\ell+\alpha_j\big)^{\otimes \dn} - \alpha_j^{\otimes \dn}
 +\big((m+1)\alpha_\ell+\alpha_j\big)^{\otimes \dn} - \big(m\alpha_\ell+\alpha_j\big)^{\otimes \dn} \big) \\
&= \frac{1}{2}\sum_{\nu=0}^{d-1}\big(e_{m,\dn-\nu}+f_{m,\dn-\nu}+(-1)^{\dn-\nu}\big)\gamma_\nu
\end{align*}
and similarly
\[ w_m = \frac{1}{2}\sum_{\nu=0}^{d-1}\big(e_{m,\dn-\nu}+f_{m,\dn-\nu}+(-1)^{\dn-\nu+1}\big)\gamma_\nu. \]
Thus we get
\begin{gather*}
\qR_m = (1-\bich_\qq(w_m))(1-\bich_\qq(v_m)) \\
\hphantom{\qR_m }{}= 1
-\prod_{\nu=0}^{\dn-1} q_\nu^{\frac{1}{2}\big(e_{m,\dn-\nu}+f_{m,\dn-\nu}+(-1)^{\dn-\nu}\big)}
-\prod_{\nu=0}^{\dn-1} q_\nu^{\frac{1}{2}\big(e_{m,\dn-\nu}+f_{m,\dn-\nu}+(-1)^{\dn-\nu+1}\big)} \\
\hphantom{\qR_m=}{}+\prod_{\nu=0}^{\dn-1} q_\nu^{e_{m,\dn-\nu}+f_{m,\dn-\nu}} \\
\stackrel{\text{\eqref{ef1}, \eqref{ef2}}}{=}
 \left( \prod_{\nu=0}^{d-2} q_\nu^{e_{m,\dn-\nu}} \right)
\left( 1-\prod_{\nu=0}^{\dn-1} q_\nu^{\frac{1}{2}\big(e_{m-1,\dn-\nu}+f_{m-1,\dn-\nu}+(-1)^{\dn-\nu}\big)} \right)
\\
\hphantom{\qR_m=}{}\times\left( 1-\prod_{\nu=0}^{\dn-1} q_\nu^{\frac{1}{2}\big(e_{m-1,\dn-\nu}+f_{m-1,\dn-\nu}+(-1)^{\dn-\nu+1}\big)} \right)
 \\
\hphantom{\qR_m=}{}+ \left(1-\prod_{\nu=0}^{d-2} q_\nu^{e_{m,\dn-\nu}} \right) \left(1-\prod_{\nu=0}^{d-1} q_\nu^{f_{m,\dn-\nu}} \right) \\
\hphantom{\qR_m}{}=r_m \qR_{m-1}+(1-r_m)(1-z_m).\tag*{\qed}
\end{gather*}\renewcommand{\qed}{}
\end{proof}

\begin{Example}\quad
\begin{enumerate}\itemsep=0pt
\item If $\dn=2$, then we recover the recursion \eqref{rosso_rec} and Rosso's condition from Nichols algebras (notice that then $r_m=q_0=q_{\ell,\ell}$ for all $m$).
\item If $\dn=3$, then
\begin{gather*}
\chi_\qq(v_m) = q_{\ell\ell\ell}^{3\frac{m(m+1)}{2}}(q_{\ell\ell j}q_{\ell j\ell}q_{j \ell\ell})^{m+1}
= q_{0}^{3\frac{m(m+1)}{2}}q_{1}^{m+1}, \\
\chi_\qq(w_m) = q_{\ell\ell\ell}^{3\frac{m(m+1)}{2}+1}(q_{\ell\ell j}q_{\ell j\ell}q_{j \ell\ell})^{m}
q_{\ell j j}q_{j \ell j}q_{j j \ell}
= q_{0}^{3\frac{m(m+1)}{2}+1}q_{1}^{m} q_{2}.
\end{gather*}
\item $d=4$: With respect to the basis $\gamma_0,\dots,\gamma_4$,
\begin{gather*}
v_m = \left( 2m^3+3m^2+2m+1,
\frac{3}{2}m^2+\frac{3}{2}m,
m+1,
0,
0\right), \\
w_m = \left(2 m^3+3 m^2+2m,
\frac{3}{2}m^2+\frac{3}{2}m+1,
m,
1,
0
\right).
\end{gather*}
\end{enumerate}
\end{Example}

\subsection{Higher Rosso conditions}
Remember that the original formula by Rosso \eqref{rosso_cartan} in multiplicative notation includes a geometric series which we replaced by the condition that $q_{\ell,\ell}^{m+1}=1$ and $q_{\ell,\ell}\ne 1$. In other words, when $d=2$, we have to divide $\qR_m$ by $1-q_{\ell,\ell}$. Moreover, this division is necessary to ensure Axiom (M2) (see Proposition~\ref{RdCartan}) if we aim to get a Weyl groupoid.

So we need to find some polynomial $\tilde g_m$ in $q_0,\dots,q_d$ such that $\qR_m/\tilde g_m$ is a condition that produces a Weyl groupoid. In particular, to obtain a Cartan matrix and a Cartan graph (which will be done in Definition~\ref{cartan_d} and Proposition~\ref{RdGraph}),
$\qR_m/\tilde g_m$ should not be rational functions with non trivial denominators.

In the classical case of Nichols algebras, $\tilde g_m$ is a divisor of $(1-\chi_\qq(v_m))$. For $d>2$, the polynomial $(1-\chi_\qq(v_m))$ in $q_0,\dots,q_d$ does not factorize in a nice way. However, we notice that the greatest common divisor of the coordinates of $v_m$ (as polynomials in $m$) is $m+1$. Indeed,
\begin{align*}
v_m &= \frac{1}{2}\sum_{\nu=1}^d\big((m+1)^\nu-m^\nu+(-1)^\nu\big)\gamma_{d-\nu} \\
&=\frac{1}{2}\sum_{\nu=1}^d (m+1) \left( (m+1)^{\nu-1}+\sum_{k=0}^{\nu-1} (-1)^{\nu-k} m^k \right)\gamma_{d-\nu}.
\end{align*}
This suggests the following definition.

\begin{Definition}
With the above notation and $m\in\NN$ we set
\begin{gather}
s_m :=  \sum_{\nu=1}^d \frac{1}{2}\left( (m+1)^{\nu-1}+\sum_{k=0}^{\nu-1} (-1)^{\nu-k} m^k \right) \gamma_{\dn-\nu}, \nonumber\\
g_m :=  \chi_\qq(s_m) = \prod_{\nu=1}^d q_{\dn-\nu}^{\frac{1}{2}\left( (m+1)^{\nu-1}+\sum_{k=0}^{\nu-1} (-1)^{\nu-k} m^k \right)}, \nonumber\\
\label{Rossod} R_m :=  \frac{\qR_m}{1-g_m}
= \frac{(1-\bich_\qq(v_m))(1-\bich_\qq(w_m))}{1-\bich_\qq(s_m)}.
\end{gather}
\end{Definition}

Unfortunately, some of the coordinates of $s_m$ are in $\frac{1}{2}+\ZZ$ (the exponents of $q_i$ for $m$ odd, $(i+d)$ odd, and $i<\dn-1$) and we possibly get square roots of $q_0,\dots,q_\dn$ in $g_m$. In order to obtain a consistent condition and definition, we may fix square roots $\sqrt{q_0},\dots,\sqrt{q_\dn}$ first and define the $q_i$ to be their squares. This corresponds to enlarging the $\ZZ$-module considered in additive notation. We will see below that in most cases, the choice of square roots in fact does not affect the condition $R_m=0$.

\begin{Corollary} The conditions $R_m$, $m\in\NN_0$ satisfy:
\begin{gather*}
R_m \in   \ZZ\big[\sqrt{q_0},\dots,\sqrt{q_d}\big],\nonumber\\
R_0 =  1-z_0, \nonumber\\
R_m =  \frac{r_m(1-g_{m-1})}{1-g_m} R_{m-1} + \frac{(1-r_m)(1-z_m)}{1-g_m}.
\end{gather*}
\end{Corollary}

\begin{proof}
Since $\chi_\qq(v_m)=g_m^{m+1}$, the polynomial $1-\chi_\qq(v_m)$ is divisible by $1-g_m$, the quotient being a geometric series. However, these are integral polynomials in $\sqrt{q_i}$, $i=0,\dots,d$ by the definition of $g_m$.
One can check that
\[g_0=\prod_{\nu=1}^d q_{\dn-\nu}^{\frac{1}{2}(1+(-1)^\nu)}=r_0,
\]
thus
$R_0 = (1-r_0)(1-z_0)/(1-g_0)=1-z_0$.
The recursion is a direct consequence of Theo\-rem~\ref{rosso_rec_d}.
\end{proof}

\begin{Remark}The condition $R_m=0$ may be written as
\begin{gather*}
\big(\bich_\qq(v_m)=1 \quad \text{and}\quad \bich_\qq(s_m)\ne 1\big)
\qquad \text{or} \qquad
\bich_\qq(w_m)=1,
\end{gather*}
which is equivalent to
\begin{gather*}
\big(\bich_\qq(s_m)^{m+1}=1 \quad \text{and}\quad \bich_\qq(s_m)\ne 1\big)
\qquad \text{or} \qquad
\bich_\qq(w_m)=1.
\end{gather*}
Note that square roots of $q_i$ only appear in $\bich_\qq(s_m)$ for odd $m$ since $v_m\in\ZZ[q_0,\dots,q_\dn]$ and $s_m=v_m/(m+1)$.

Now assume that $\bich_\qq(s_m)^2=1$. Then the first condition cannot be satisfied for even $m$.
If~$m$~is odd,
then $\bich_\qq(s_m)^{m+1}=1^{(m+1)/2}=1$ and in this case the first condition is satisfied if and only if $\bich_\qq(s_m)=-1$. Hence $R_m=0$ is equivalent to
\begin{gather*}
\big(\bich_\qq(v_m)=1 \quad\text{and} \quad\bich_\qq(2s_m)\ne 1\big)
\qquad \text{or} \qquad
(\bich_\qq(s_m)=-1 \quad \text{and}\quad m \text{ odd})\nonumber \\
\text{or} \qquad
\bich_\qq(w_m)=1.
\end{gather*}

It may happen that $\bich_\qq(s_m)=-1$ with $m$ odd never occurs in the construction of the Weyl groupoid of a specific tensor. In this case, the general Rosso condition would be
\begin{gather*}
\big(\bich_\qq(v_m)=1 \quad \text{and}\quad \bich_\qq(2s_m)\ne 1\big)
\qquad \text{or} \qquad
\bich_\qq(w_m)=1,
\end{gather*}
and it would avoid the square roots.
However, in the classical case ($d=2$) one often has $q_{\ell,\ell}=-1$, and we expect that this
condition would exclude many cases where $s_m\in\ZZ^d$ and $\bich_\qq(s_m)=-1$.
\end{Remark}

\subsection{Weyl groupoids from higher tensor powers}
We may now use the Rosso condition $R_m$ for $\dn\ge 2$ to compute Weyl groupoids for a given $\qq$ in the same way as for $\dn=2$.
For $\ell,j\in \{1,\dots,n\}$ we write $R^{\ell,j}_m:=R_m$.

\begin{Definition}\label{cartan_d}
Let $\dn\in\NN$ and
$\qq = (q_{i_1,\dots,i_\dn})_{i_1,\dots,i_\dn}\in \KK^{\{1,\dots,n\}^\dn}$.
For each tuple $(i_1,\dots,i_\dn)$, we fix a square root
$\sqrt{q_{i_1,\dots,i_\dn}}\in\KK$ of $q_{i_1,\dots,i_\dn}$ and
by abuse of notation we write
\[\sqrt{\qq}:=\big(\sqrt{q_{i_1,\dots,i_\dn}}\big)_{i_1,\dots,i_\dn}.\]
Then using equation \eqref{Rossod} we define
\begin{equation*}
c^\qq_{\ell,j} = \begin{cases} -\min\big\{ m\in\NN_0 \mid R^{\ell,j}_m = 0 \big\}, & \ell\ne j, \\
2, & \ell = j, \end{cases}
\end{equation*}
and obtain a matrix $c^\qq=\big(c^\qq_{\ell,j}\big)_{\ell,j}$ if such an $m$ exists for all $\ell\ne j$.
In this case, we call $c^\qq$ the {\it Cartan matrix} of $\qq$.
It defines maps $\sigma_\ell$ via
\[ \sigma_\ell(\qq)_{i_1,\dots,i_\dn} :=
\prod_{k_1,\dots,k_\dn=1}^n q_{k_1,\dots,k_\dn}^{\sigma_\ell(\alpha_{i_1})_{k_1}\cdots \sigma_\ell(\alpha_{i_\dn})_{k_\dn}}. \]
These maps act on the square roots in a compatible way since additively, $\qq$ is just twice the vector $\sqrt{\qq}$ and $\sigma_\ell^{\otimes \dn}$ is linear.
\end{Definition}

\begin{Proposition}\label{RdCartan}
Let $\dn\in \NN$ be even and
$\qq = (q_{i_1,\dots,i_\dn})_{i_1,\dots,i_\dn}\in \KK^{\{1,\dots,n\}^\dn}$.
If a Cartan matrix~$c^\qq$ is defined, then it is a generalized Cartan matrix.
\end{Proposition}
\begin{proof}
Axiom (M1) is satisfied by definition.
For (M2), remember that $R_0 = 1-z_0$. Since
\[z_0 = \sum_{\nu=0}^{d-1} \frac{1}{2}\big(1-(-1)^\nu\big)\gamma_\nu
= \sum_{\nu=0}^{d-1} \frac{1}{2}\big(1-(-1)^{d-\nu}\big)\gamma_\nu, \]
the condition $R_0$ is the same for $c^\qq_{\ell,j}$ as for $c^\qq_{j,\ell}$.
\end{proof}

\begin{Proposition}\label{RdGraph}
Let $\dn\in \NN$ be even, $I=\{1,\dots,n\}$, and
$A$ be a set of tensors $\qq\in \KK^{I^\dn}$ with defined Cartan matrix $c^\qq$.
Assume that for all $\qq\in A$ and $\ell\in I$ we have
\[ \rho_\ell(\qq):=\sigma^{\qq}_\ell (\qq) \in A. \]
Then $(I,A,(\rho_i)_{i\in I},(c^\qq)_{\qq\in A})$ is a Cartan graph.
\end{Proposition}
\begin{proof}
Let $\qq\in A$ and $\ell,j \in I$.
The matrices $c^\qq$ are generalized Cartan matrices by Proposition~\ref{RdCartan}.
If Axiom (C2) is satisfied, then (C1) is satisfied as well because $\rho_\ell$ is defined using a reflection which has order two and $\sigma^{\qq}_\ell = \sigma^{\rho_\ell(\qq)}_\ell$ since the relevant Cartan entries are equal by~(C2).

For (C2), we need to compare $c^\qq_{\ell,j}$ and $c^{\rho_\ell(\qq)}_{\ell,j}$ for $\ell\ne j$.
With $m:=-c^\qq_{\ell,j}$, we have $R_m=0$ and $R_k\ne 0$ for $k<m$.
Now $R_m = (1-\bich_\qq(v_m))(1-\bich_\qq(w_m))/(1-\chi_\qq(s_m))$.
Recall that $v_m$ and $w_m$ are eigenvectors, i.e., $\big(\sigma^{\qq}_\ell\big)^{\otimes \dn}(v_m) = v_m$,
$\big(\sigma^{\qq}_\ell\big)^{\otimes \dn}(w_m) = -w_m$;
also, $\big(\sigma^{\qq}_\ell\big)^{\otimes \dn}(s_m) = s_m$ because $(m+1)s_m=v_m$.
Hence replacing $\qq$ by $\sigma^{\qq}_\ell (\qq)$ will map $R_m$ to
$(1-\bich_\qq(v_m))(1-\bich_\qq(-w_m))/(1-\chi_\qq(s_m))$.
But $1-\bich_\qq(-w_m)=0$ if and only if $-w_m=0$ or equivalently $w_m=0$.
Therefore, via~\eqref{cartan_d} the new condition produces the same Cartan entry
$c^\qq_{\ell,j}=c^{\rho_\ell(\qq)}_{\ell,j}$ for $\ell\ne j$.
\end{proof}

We collect our results in the following theorem.

\begin{Theorem}\label{wg_from_q}
Let $\dn\in \NN$ be even, $I=\{1,\dots,n\}$, $\qq_0\in \KK^{I^\dn}$ a tensor, and fix square roots~$\sqrt{\qq_0}$.
Let $A$ be the smallest set with $\qq_0\in A$ and $\rho_\ell(\qq):=\sigma^{\qq}_\ell (\qq)\in A$ for all $\qq\in A$, $\ell\in I$,
assuming that $c^\qq$ is defined for all $\qq\in A$.
Then
$\Cc:=\big(I,A,(\rho_i)_{i\in I},\big(c^\qq\big)_{\qq\in A}\big)$ is a Cartan graph.
\\
We call $\Wg(\Cc)$ the {\it Weyl groupoid} of $\qq_0$.
\end{Theorem}

\begin{proof}
This is a special case of Proposition~\ref{RdGraph}.
\end{proof}

\subsection{Examples in rank two}

Finite Weyl groupoids of rank two are in bijection with
the set of triangulations of convex polygons by non-intersecting diagonals:

Let $\qq$ be a tensor of rank two ($n=2$, $I=\{1,2\}$) and consider the sequence of Cartan entries
\[c := \bigl(-c^\qq_{1,2},-c^{\sigma^\qq_1(\qq)}_{2,1},-c^{\sigma^{\sigma_1(\qq)}_2(\sigma^\qq_1(\qq))}_{1,2},\dots\bigr). \]
If the Weyl groupoid is finite, then this sequence $c$ may be interpreted as the numbers of triangles at the vertices of a triangulation of a convex polygon by non-intersecting diagonals. It is called a {\it quiddity cycle}.

Writing the coordinates of the sets $R^a$ of a root system into a matrix yields a so called {\it frieze pattern},\footnote{We omit a definition since we do not need it here.} see \cite{p-C14}. Figure~\ref{frieze} shows the example of the sequence $c=(1,4,1,2,2,2)$.
\begin{figure}[h]
 \centering
 \begin{minipage}[b]{0.44\textwidth}
 $ \begin{array}{cccccccccccc}
 & & \ddots & & & & & & & & & \\
 0 & 1 & 1 & 3 & 2 & 1 & 0 & & & & & \\
& 0 & 1 & 4 & 3 & 2 & 1 & 0 & & & & \\
 & & 0 & 1 & 1 & 1 & 1 & 1 & 0 & & & \\
 & & & 0 & 1 & 2 & 3 & 4 & 1 & 0 & & \\
 & & & & 0 & 1 & 2 & 3 & 1 & 1 & 0 & \\
 & & & & & 0 & 1 & 2 & 1 & 2 & 1 & 0 \\
 & & & & & & & & & \ddots & &
\end{array}
$
 \end{minipage}
 \qquad \qquad
 \begin{minipage}[b]{0.32\textwidth}
 \begin{tikzpicture}[auto,baseline=(s.center)]
 \node[name=s, draw, shape=regular polygon, regular polygon sides=6, minimum size=2.8cm] {};
 \draw (s.corner 1) to (s.corner 3);
 \draw (s.corner 6) to (s.corner 3);
 \draw (s.corner 5) to (s.corner 3);
 \draw[shift=(s.corner 1)] node[above] {{\small 2}};
 \draw[shift=(s.corner 2)] node[above] {{\small 1}};
 \draw[shift=(s.corner 3)] node[left] {{\small 4}};
 \draw[shift=(s.corner 4)] node[below] {{\small 1}};
 \draw[shift=(s.corner 5)] node[below] {{\small 2}};
 \draw[shift=(s.corner 6)] node[right] {{\small 2}};
 \end{tikzpicture}
 \end{minipage}
\caption{A frieze pattern and the corresponding triangulation.\label{frieze}}
\end{figure}
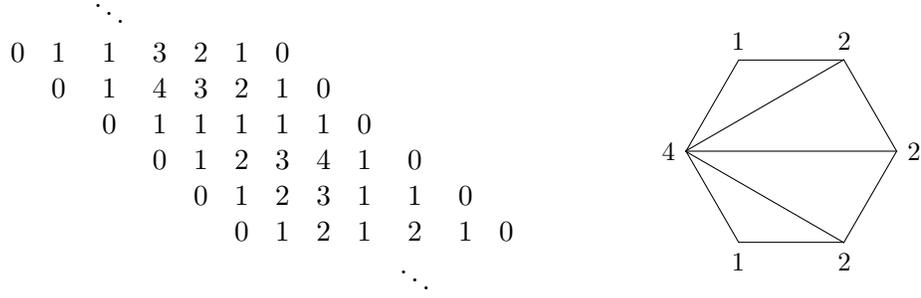

We have many examples of tensors producing interesting Weyl groupoids. Since we do not have a classification yet, we content ourselves with two examples.

\begin{Example}
Let $d=4$, $\zeta$ primitive $11$-th root of unity, $\mu:=-\zeta$, and
\[ \qq=(q_0,\dots,q_4)=\big(\zeta^2,\zeta^2,\zeta^2,\zeta^2,\zeta^2\big), \qquad \sqrt{\qq}=\big(\sqrt{q_0},\dots,\sqrt{q_4}\big) = (\mu,\mu,\mu,\mu,\mu). \]
We compute the values
\[\renewcommand{\arraystretch}{1.2}
\begin{array}{l|rrrrrr}
m & 0 & 1 & 2 & 3 \\
\hline
\chi_\qq(v_m) & \mu^4 & \mu^4 & \mu^2 & 1 \\
\chi_\qq(w_m) & \mu^4 & \mu^4 & \mu^2 & 1 \\
\chi_\qq(s_m) & \mu^4 & \mu^{13} & \mu^8 & \mu^{11}
\end{array}
\]
so the smallest $m$ with $R_m^{1,2}=0$ is $3$ and $c^\qq_{1,2}=-3$.
Reflecting produces the new ``tensor''
\[\sigma_1^{\qq}\big(\sqrt{\qq}\big) = \big(\mu,\mu^{9},\mu^{20},\mu^{12},\mu^{11}\big).\]
To compute the next Cartan entry we need to reverse this sequence because we now want to reflect at label $2$, the new $q_i$ are
$\big(\sqrt{q_0},\dots,\sqrt{q_4}\big)=\big(\mu^{11},\mu^{12},\mu^{20},\mu^{9},\mu\big)$:
\[\renewcommand{\arraystretch}{1.2}
\begin{array}{l|rrrrrr}
m & 0 & 1 \\
\hline
\chi_{\sigma_1^{\qq}}(v_m) & \mu^{18} & \mu^{20} \\
\chi_{\sigma_1^{\qq}}(w_m) & \mu^{20} & 1 \\
\chi_{\sigma_1^{\qq}}(s_m) & \mu^{18} & \mu^{10}
\end{array}
\]
The next Cartan entry is $-1$. Continuing this procedure, we obtain the complete sequence of~Cartan entries.
The sequence of negative Cartan entries is the quiddity cycle of the corresponding Weyl groupoid:
\[ (3, 1, 2, 3, 2, 1, 3). \]
As a triangulation of a $7$-gon, this is,
\begin{center}
\includegraphics[width=0.16\textwidth]{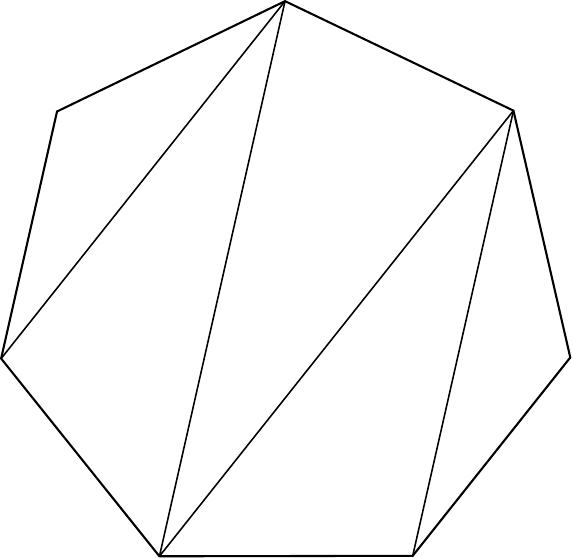}
\end{center}
This is not a Weyl groupoid of a Nichols algebra of diagonal type.
\end{Example}

\begin{Example}
Let $d=4$, $\zeta$ primitive $7$-th root of unity, $\mu:=-\zeta$, and
\[ \qq=(q_0,\dots,q_4)=\big(\zeta,\zeta^2,\zeta,\zeta^2,\zeta^2\big), \qquad \big(\sqrt{q_0},\dots,\sqrt{q_4}\big) = \big(\mu^4,\mu,\mu^4,\mu,\mu\big). \]
Then the (negative) Cartan entries ({\it quiddity cycle}) of the corresponding Weyl groupoid are
\[ (2, 1, 5, 1, 3, 1, 5, 1, 2, 3). \]
The triangulation is
\begin{center}
\includegraphics[width=0.2\textwidth]{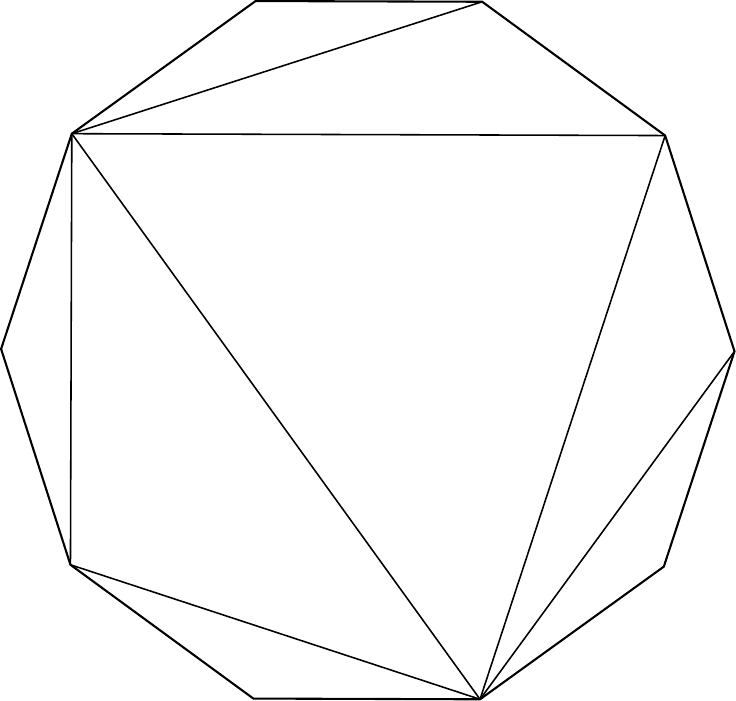}
\end{center}
Again, this is a quiddity cycle that does not appear in the classical theory of Nichols algebras of diagonal type.
\end{Example}

\begin{conje}
Any finite Weyl groupoid of rank two is obtained as the Weyl groupoid of some $\qq\in \CC^{\{1,2\}^\dn}$, $\dn\in\NN$.
\end{conje}

It is thus interesting to find a mathematical object similar to a Nichols algebra that would yield the recursion
of Theorem \ref{rosso_rec_d}
for $\dn>2$. This question is the subject of the next section; the idea will be to consider higher commutativity constraints.

\section{Higher tensors from abelian cohomology theory}\label{ab_coh}

In this section, we want to give a motivation for the construction in Section \ref{sec: tensor powers} in terms of higher commutativity constraints. As in the classical case ($d=2$), a higher commutativity constraint (or \textit{braiding}) will be an abelian $(2d-1)$-cocycle. In the diagonal case, these are simply determined by tensors
\[\qq = (q_{i_1,\dots,i_\dn})_{i_1,\dots,i_\dn}\in \KK^{\{1,\dots,n\}^\dn}\]
as in Section \ref{generalized_Rosso}. The numbers
\begin{align*}
 q_k :=\prod_{\substack{i_1,\dots,i_\dn\in\{\ell,j\}\\ |\{\nu\:\mid\: i_\nu=j\}|=k}} q_{i_1,\dots,i_\dn},
\end{align*}
from Definition~\ref{def: defn of qk's} can then be reinterpreted as invariants with respect to $(2d-1)$st abelian cohomology. Before we introduce higher commutativity constraints, we want to revisit the classical case in order to see where the third abelian cohomology (which we assume to be known) comes into play.

\subsection{Nichols algebras as graded vector spaces}

Let $G$ be an abelian group and let $\mathsf{Vect}_G$ denote the category of finite-dimensional $G$-graded vector spaces over a field $\KK$.

It is well-known \cite[Example~8.4.8]{b-EGNO-15} that braided monoidal structures on $\mathsf{Vect}_G$ are classified (up to braided monoidal equivalence) by the third abelian cohomology $H^3_1\big(G,\KK^\times\big)$. The construction goes as follows. Let $(\omega,\theta) \in Z^3_1\big(G,\KK^\times\big)$ be an abelian 3-cocycle. For $g,h,k \in G$, let~$V_g$,~$V_h$ and~$V_k$ be the corresponding simple objects in $\mathsf{Vect}_G$. Then we can define an associator~$a$ and a braiding $c$ on $\mathsf{Vect}_G$ as follows:
\begin{align*}
 a_{g,h,k}\colon \ (V_g\otimes V_h) \otimes V_k &\longrightarrow V_g\otimes (V_h \otimes V_k),\\
 \qquad e_g\otimes e_h \otimes e_k &\longmapsto \omega(g,h,k) \, e_g\otimes e_h \otimes e_k,\\
 c_{g,h}\colon \ V_g \otimes V_h &\longrightarrow V_h \otimes V_g,\\
 e_g\otimes e_h &\longmapsto \theta( g,h)\, e_h \otimes e_g.
\end{align*}
Here, the tensor product of simple objects $V_g$ and $V_h$ is simply given by $V_g \otimes V_h := V_{gh}$. We denote the resulting braided monoidal category by $\mathsf{Vect}_G^{(\omega,\theta)}$.

If the (ordinary) 3-cocycle $\omega$ is trivial, then we can read off Table \ref{tab:generators of abelian chain complex} ($n=4$, $k=1$) that $\theta$ must be a bicharacter. Hence, we can define a $G$-action on every simple object $V_g \in \mathsf{Vect}_G^{(1,\theta)}$ by setting $h.e_g:=\theta(g,h)\,e_g$. For this reason, any object in $\mathsf{Vect}_G^{(1,\theta)}$ can be regarded as a~Yetter--Drinfeld module over $G$, so that we are in the setting of Section \ref{sec: Nichols algebras}. In particular, for $g_1, \dots, g_n \in G$ the object
\begin{align*}
 V=\oplus_{i=1}^n\,V_{g_i} \in \mathsf{Vect}_G^{(1,\theta)}
\end{align*}
is a diagonally braided vector space with braid matrix
\begin{align*}
 \qq = (q_{i,j})_{i,j} := (\theta(g_i,g_j))_{i,j}
\end{align*}
and we can define the Nichols algebra $\BB(V)$ of~$V$. However, it seems natural to us to think of~$\BB(V)$ as a Hopf algebra inside the braided monoidal category $\mathsf{Vect}_G^{(1,\theta)}$ because its construction (including the operators $\partial$ and $\ad$) completely takes place in $\mathsf{Vect}_G^{(1,\theta)}$.
As pointed out by a~referee, this has also been done recently in \cite[Section~5.2]{MR4498166}.

\subsection{Abelian cohomology theory}

In this subsection, we recapitulate the cohomology theory of abelian groups first introduced by Eilenberg and MacLane in \cite{Eilenberg1950, Eilenberg1950a, MacLane1952}. They realized that the homology theory appropriate to an \textit{abelian} group\footnote{We write $\Pi$ instead of $G$ here because this is the usual symbol for abelian cohomology.} $\Pi$ is not given by the ordinary bar complex $A_\bullet(\Pi)$ of the group ring $\ZZ[\Pi]$, as the proof of the condition $\partial \partial=0$ only uses associativity of $\Pi$ but not commutativity. In order to resolve this issue, they defined further cell complexes $A^k_\bullet(\Pi)$, where $k \in \ZZ_{\geq 0}$ can be thought of as the \textit{level of commutativity} which is taken into account. In particular, we have $A_\bullet(\Pi)=A^0_\bullet(\Pi)$.

We will reproduce the inductive definition of the cell complexes $A^k_\bullet(\Pi)$ from \cite{Eilenberg1950a}:

Let $A_\bullet(\Pi)$ be the ordinary bar complex of the abelian group $\Pi$.

More precisely, $A_n(\Pi)$ is the free $\ZZ$-module generated by $n$-tuples $[x_1,\dots,x_n]$ of elements $x_i \in \Pi$. The boundary operators $\partial_n\colon A_n(\Pi) \to A_{n-1}(\Pi)$ are given by
\begin{align*}
 \partial_n [x_1,\dots,x_n] := [x_2, \dots,x_n] + \sum_{i=1}^{n-1}\,(-1)^i [x_1, \dots,x_{i}x_{i+1},\dots,x_n] + (-1)^n[x_1,\dots,x_{n-1}].
\end{align*}

Since this is the basis of our induction, we set
\begin{align*}
 A^0_\bullet(\Pi)=A_\bullet(\Pi), \qquad \partial^0_n = \partial_n \qquad \text{and} \qquad [x_1 \, |_0 \dots |_0 \, x_n]=[x_1, \dots , x_n].
\end{align*}
As a free $\ZZ$-module, $A^k_n(\Pi)$ is generated by $p$-tuples $[\alpha_1 \, |_k \dots |_k \, \alpha_p]$, where the $\alpha_i$ are generators of $A^{k-1}_{n_i}(\Pi)$, such that
\begin{align*}
 n= \sum_{i=1}^p \,n_i + (p-1)k.
\end{align*}
Hence, we can think of the symbol $|_k$ as a placeholder for $k$ arguments. In particular, we have a stabilizing chain
\begin{align}\label{eq: chain of cell complexes}
 A^0_n(\Pi) \leq A^1_n(\Pi) \leq \dots \leq A^{n-2}_n(\Pi) = A^{n-1}_n(\Pi) = \cdots
\end{align}
of free $\ZZ$-modules.

The graded $\ZZ$-module $A^k_\bullet(\Pi)=\bigoplus_{n=0}^\infty \, A^k_n(\Pi)$ is endowed with a so-called {\it $k$-shuffle product}
\begin{align*}
 \ast_k\colon \ A^k_n(\Pi) \times A^k_m(\Pi) \to A^k_{n+m+k}(\Pi).
\end{align*}
For $\alpha=[\alpha_1 \, |_k \dots |_k \, \alpha_p] \in A^k_n(\Pi)$ and $\beta =[\beta_1 \, |_k \dots |_k \, \beta_q] \in A^k_m(\Pi)$, a~shuffle of $\alpha$ and $\beta$ is a~$(p+q)$-tuple $\gamma=[\gamma_1 \, |_k \dots |_k \, \gamma_{p+q}] \in A^k_{n+m+k}(\Pi)$ obtained by permuting the $(p+q)$-tuple
\begin{align*}
 [\alpha_1 \, |_k \dots |_k \, \alpha_p \, |_k \,\beta_1 \, |_k \dots |_k \, \beta_q]
\end{align*}
in a way that preserves the order of the $\alpha_i$'s and $\beta_j$'s, respectively. To such a shuffle we can associate a number
\begin{align*}
 \epsilon_\gamma:= \sum_{\text{$(i,j)$ such that $\alpha_i$ is after $\beta_j$ in $\gamma$}} \, (n_i+k)(n_j+k).
\end{align*}
The $k$-shuffle product on $A^k_\bullet(\Pi)$ is then defined as follows:
\begin{align*}
 [\alpha_1 \, |_k \dots |_k \, \alpha_p] \ast_k [\beta_1 \, |_k \dots |_k \, \beta_q] := \sum_{\text{shuffles $\gamma$ of $\alpha$ and $\beta$}} \, (-1)^{\epsilon_\gamma}[\gamma_1 \, |_k \dots |_k \, \gamma_{p+q}].
\end{align*}

Having introduced the $k$-shuffle product, we are now ready to define the boundary operators $\partial^k_n\colon A^k_n(\Pi) \to A^k_{n-1}(\Pi)$:
\begin{align}
 \partial^k_n[\alpha_1 \, |_k \dots |_k \,\alpha_p]:=&{} \sum_{i=1}^p \, (-1)^{a_{i-1}}[\alpha_1 \, |_k \dots |_k \, \partial^{k-1}\alpha_i \, |_k \dots |_k \, \alpha_p]\nonumber\\
 &{}+\sum_{i=1}^{p-1}\,(-1)^{a_i}[\alpha_1 \, |_k \dots |_k \,\alpha_{i-1} \, |_k \, \alpha_i \ast_{k-1} \alpha_{i+1} \, |_k \, \alpha_{i+2} \, |_k \dots |_k \, \alpha_p],\label{def: boundary operator of abelian chains}
\end{align}
where
\[
a_i=\sum_{j=1}^i  n_j + i\cdot k.
\] For the proof that $\partial^k_{n-1}\partial^k_n = 0$ holds, we refer to \cite{Eilenberg1950a}.

We refer to the chain complex $A^k_\bullet(\Pi)$ as the {\it abelian complex of $\Pi$ of commutativity level $k$}.

\begin{conve}
 For an abelian group $A$, we will denote the cochain complex $\Hom(A^k_\bullet(\Pi),A)$ by $A_k^\bullet(\Pi,A)$ with coboundary operators $\delta_k^{n-1}:=\big(\partial_n^k\big)^*$. As usual, for the subgroups of (co)cycles and (co)boundaries we replace the letter $A$ by the letters $Z$ and $B$, respectively. For the homology and cohomology groups we use the notation $H^k_\bullet(\Pi)$ and $H_k^\bullet(\Pi,A)$, respectively.
\end{conve}

\begin{Example}\allowdisplaybreaks
In this example, we list generators of $A^k_n(\Pi)$ and their boundaries for $n \in \{1,2,3,4,5,6\}$, since we are going to use them repeatedly in the following section. Moreover, we believe that an explicit list can be helpful for other people dealing with abelian chain complexes independently from our motivation.

Due to the inclusions in sequence \eqref{eq: chain of cell complexes} it is sufficient to only specify the new generators appearing for every $k \in \ZZ_{\geq 0}$.

\renewcommand{\arraystretch}{1.1}
\begin{longtable}{llll}
\toprule[1.5pt]
$n$ & $k$ & Generator & Boundary \\
\midrule[1.0pt]
1 & -- &$[a]$ & $\partial [a]=0$ \\
\midrule[1.5pt]
2 & 0 & $[a,b]$ & $\partial [a,b]=[a] - [ab] + [b]$ \\
\midrule[1.5pt]
3 & 0 & $[a,b,c]$ & $\partial [a,b,c]=[b,c] - [ab,c] + [a,bc]-[a,b]$ \\
\midrule[1pt]
 & 1 & $[a|b]$ & $\partial [a|b]=[a,b] - [b,a]$ \\
\midrule[1.5pt]
4 & 0 & $[a,b,c,d]$ &
		 $\!\begin{aligned}[t]
		 \partial [a,b,c,d]={}&[b,c,d] - [ab,c,d] + [a,bc,d]\\
		 &{}-[a,b,cd] + [a,b,c]
		 \end{aligned}$ \\
\midrule[1pt]
 & 1 & $[a,b|c]$ &
		 $\!\begin{aligned}[t]
 \partial [a,b|c]={}&[b|c] -[ab|c] + [a|c] \\
 &{}- [a,b,c] + [a,c,b] - [c,a,b]
 \end{aligned}$\\
\midrule[0.5pt]
 & & $[a|b,c]$ &
		 $\!\begin{aligned}[t]
 \partial [a|b,c]={}&[a|c] - [a|bc] + [a|b] \\
 &+ [a,b,c] - [b,a,c] + [b,c,a]
 \end{aligned}$\\
\midrule[1pt]
 & 2 & $[a||b]$ & $\partial [a||b]=-[a|b] - [b|a]$ \\
\midrule[1.5pt]
5 & 0 &$[a,b,c,d,e]$ &
		 $\!\begin{aligned}[t]
 \partial[a,b,c,d,e] ={}&[b,c,d,e]-[ab,c,d,e]+[a,bc,d,e] \\
 &{}-[a,b,cd,e]+[a,b,c,de]-[a,b,c,d]
 \end{aligned}$\\
\midrule[1pt]
 & 1 & $[a,b,c|d]$ &
		 $\!\begin{aligned}[t]
 \partial [a,b,c|d]={}&[b,c|d]-[ab,c|d]+[a,bc|d]-[a,b|d] \\ &{}+[a,b,c,d]-[a,b,d,c]+[a,d,b,c]-[d,a,b,c]
 \end{aligned}$\\
\midrule[0.5pt]
 & & $[a,b|c,d]$ &
		 $\!\begin{aligned}[t]
 \partial [a,b|c,d]={}& [b|c,d] - [ab|c,d] +[a|c,d]-[a,b|d]\\
 &{}+ [a,b|cd]-[a,b|c]-[a,b,c,d]\\
 &{}+[a,c,b,d]-[c,a,b,d]-[a,c,d,b]\\
 &{}+[c,a,d,b]-[c,d,a,b]
 \end{aligned}$\\
\midrule[0.5pt]
 & & $[a|b,c,d]$ &
		 $\!\begin{aligned}[t]
 \partial [a|b,c,d]={}&[a|c,d] -[a|bc,d] + [a|b,cd] - [a|b,c] \\
 &{}+ [a,b,c,d] - [b,a,c,d] + [b,c,a,d] - [b,c,d,a]
 \end{aligned}$\\
\midrule[0.5pt]
 & & $[a|b|c]$ &
		 $\!\begin{aligned}[t]
 \partial [a|b|c]&=[a,b|c]-[b,a|c]+[a|b,c]-[a|c,b]
 \end{aligned}$\\
\midrule[1pt]
 & 2 & $[a,b||c]$ &
		 $\!\begin{aligned}[t]
		 \partial [a,b||c]=[b||c] - [ab||c] + [a||c]+[a,b|c]+[c|a,b]
		 \end{aligned}$\\
\midrule[0.5pt]
 & & $[a||b,c]$ &
		 $\!\begin{aligned}[t]
		 \partial [a||b,c]=-[a||c] + [a||bc] - [a||b]-[a|b,c]-[b,c|a]
		 \end{aligned}$\\
\midrule[1pt]
 & 3 & $[a|||b]$ &
		 $\!\begin{aligned}[t]
		 \partial [a|||b]&=[a||b]-[b||a]
		 \end{aligned}$\\
\midrule[1.5pt]
6 & 0 & $[a,b,c,d,e,f]$ &
		 $\!\begin{aligned}[t]
 \partial[a,b,c,d,e,f] ={}&[b,c,d,e,f]-[ab,c,d,e,f]+[a,bc,d,e,f]\\
 &{}-[a,b,cd,e,f]+[a,b,c,de,f]\\
&{}-[a,b,c,d,ef]+[a,b,c,d,e]
 \end{aligned}$\\
\midrule[1pt]
 & 1 & $[a,b,c,d|e]$ &
		 $\!\begin{aligned}[t]
 \partial[a,b,c,d|e] ={}& [b,c,d|e] - [ab,c,d|e] + [a,bc,d|e]-[a,b,cd|e] \\
 &{}+ [a,b,c|e]-[a,b,c,d,e]+[a,b,c,e,d] \\
 &{}-[a,b,e,c,d]+[a,e,b,c,d] -[e,a,b,c,d]
 \end{aligned}$\\
\midrule[0.5pt]
 & & $[a,b,c|d,e]$ &
		 $\!\begin{aligned}[t]
 \partial[a,b,c|d,e] ={}& [b,c|d,e] - [ab,c|d,e] + [a,bc|d,e] - [a,b|d,e] \\
 &{}+[a,b,c|e] - [a,b,c|de] + [a,b,c|d]+[a,b,c,d,e] \\
 &{}- [a,b,d,c,e]+[a,d,b,c,e] - [d,a,b,c,e]\\
 &{}+[a,b,d,e,c] - [a,d,b,e,c] + [d,a,b,e,c]\\
 &{}+[a,d,e,b,c] - [d,a,e,b,c]+[d,e,a,b,c]
 \end{aligned}$\\
\midrule[0.5pt]
 & & $[a,b|c,d,e]$ &
		 $\!\begin{aligned}[t]
 \partial[a,b|c,d,e] ={}& [b|c,d,e] - [ab|c,d,e] + [a|c,d,e]-[a,b|d,e] \\
 &{}+[a,b|cd,e]-[a,b|c,de]+[a,b|c,d]-[a,b,c,d,e] \\
 &{} + [a,c,b,d,e]-[a,c,d,b,e] + [a,c,d,e,b]\\
 &{}-[c,a,b,d,e] + [c,a,d,b,e] - [c,a,d,e,b]\\
 &{}-[c,d,a,b,e] + [c,d,a,e,b]-[c,d,e,a,b]
 \end{aligned}$\\
\midrule[0.5pt]
 & & $[a|b,c,d,e]$ &
		 $\!\begin{aligned}[t]
 \partial[a|b,c,d,e] ={}& [a|c,d,e] - [a|bc,d,e] + [a|b,cd,e]-[a|b,c,de]\\
 &{}+[a|b,c,d]+[a,b,c,d,e]-[b,a,c,d,e]\\
 &{}+[b,c,a,d,e]-[b,c,d,a,e]+[b,c,d,e,a]
 \end{aligned}$\\
\midrule[0.5pt]
 & & $[a,b|c|d]$ &
		 $\!\begin{aligned}[t]
 \partial[a,b|c|d] ={}& [b|c|d]-[ab|c|d]+[a|c|d]-[a,b,c|d]\\
 &{}+[a,c,b|d]-[c,a,b|d]-[a,b|c,d]+[a,b|d,c]
 \end{aligned}$\\
\midrule[0.5pt]
 & & $[a|b,c|d]$ &
		 $\!\begin{aligned}[t]
 \partial[a|b,c|d] ={}& [a|c|d]-[a|bc|d]+[a|b|d]+[a,b,c|d]-[b,a,c|d]\\
 &{}+[b,c,a|d]-[a|b,c,d]+[a|b,d,c]-[a|d,b,c]
 \end{aligned}$\\
\midrule[0.5pt]
 & & $[a|b|c,d]$ &
		 $\!\begin{aligned}[t]
 \partial[a|b|c,d] ={}& [a|b|d] -[a|b|cd] + [a|b|c]+[a,b|c,d]-[b,a|c,d] \\
 &{}+[a|b,c,d]-[a|c,b,d]+[a|c,d,b]
 \end{aligned}$\\
\midrule[1pt]
 & 2 & $[a,b,c||d]$ &
		 $\!\begin{aligned}[t]
		 \partial [a,b,c||d] ={}& [b,c ||d] - [ab,c||d] + [a,bc||d] - [a,b||d] \\
		 &{}-[a,b,c | d] -[d|a,b,c]
		 \end{aligned}$\\
\midrule[0.5pt]
 & & $[a,b||c,d]$ &
		 $\!\begin{aligned}[t]
		 \partial [a,b||c,d] ={}& [b||c,d] - [ab||c,d] + [a||c,d]+[a,b||d]\\
		 &{}- [a,b||cd] + [a,b||c]+[a,b|c,d]-[c,d|a,b]
		 \end{aligned}$\\
\midrule[0.5pt]
 & & $[a||b,c,d]$ &
		 $\!\begin{aligned}[t]
		 \partial [a||b,c,d] ={}&-[a||c,d] + [a||bc,d] - [a||b,cd] + [a||b,c] \\
		 &{}-[a|b,c,d]-[b,c,d|a]
		 \end{aligned}$\\
\midrule[0.5pt]
 & & $[a|b||c]$ &
		 $\!\begin{aligned}[t]
		 \partial [a|b||c]&= [a,b||c] - [b,a||c]-[a|b|c]-[a|c|b]-[c|a|b]
		 \end{aligned}$\\
\midrule[0.5pt]
 & & $[a||b|c]$ &
		 $\!\begin{aligned}[t]
		 \partial [a||b|c] = -[a||b,c]+[a||c,b]-[a|b|c]-[b|a|c]-[b|c|a]
		 \end{aligned}$\\
\midrule[1pt]
 & 3 & $[a,b|||c]$ &
		 $\!\begin{aligned}[t]
		 \partial [a,b|||c] =[b|||c] - [ab|||c] + [a|||c]-[a,b||c]-[c||a,b]
		 \end{aligned}$\\
\midrule[0.5pt]
 & & $[a|||b,c]$ &
		 $\!\begin{aligned}[t]
		 \partial [a|||b,c]={}&[a|||c] - [a|||bc] + [a|||b]\\
		 &{}-[a||b,c]+[b,c||a]
		 \end{aligned}$\\
\midrule[1pt]
 & 4 & $[a||||b]$ &
		 $\!\begin{aligned}[t]
		 \partial[a||||b] = -[a||||b]-[b||||a]
		 \end{aligned}$\\
\bottomrule[1.5pt]
\caption{Generators of $A^k_n(\Pi)$ for $1 \leq n \leq 6$.}\label{tab:generators of abelian chain complex}
\end{longtable}
\end{Example}

\subsection{Invariants for abelian cohomology of commutativity level 1}

In this section, we study the complex $A^1_\bullet(\Pi)$ in more detail. In particular, we study a certain family of chains $c_\lambda \in A^1_{2d-1}(\Pi)$ indexed by compositions $\lambda$ of $d \in \NN$.

Let $\lambda = (\lambda_1, \dots, \lambda_p)$ be a $p$-composition of $d$. We define $\{\alpha_1;\dots ;\alpha_p\}_\lambda \in A_{2d-1}^1(\Pi)$ to be the sum over all generators $[\beta_1 |\dots |\beta_d] \in A_n^1(\Pi)$, where the arguments $(\beta_1, \dots,\beta_d)$ run over all multiset permutations of the multiset
\begin{align*}
 \{\alpha_1, \dots, \alpha_1,\alpha_2, \dots, \alpha_2, \dots, \alpha_p, \dots , \alpha_p \}
\end{align*}
in which every element $\alpha_i$ appears $\lambda_i$ times. In particular, $\{a_1;\dots ;a_p\}_\lambda \in A_{2d-1}^1(\Pi)$ has $\frac{d!}{\lambda_1! \cdots \lambda_p!}$ summands.

\begin{Remark}\label{rem: symmetric chains don't depend on order of composition}
 Clearly, the chains $\{\alpha_1;\dots ;\alpha_p\}_\lambda \in A_{2d-1}^1(\Pi)$ depend on the order of the $p$-composition $\lambda=(\lambda_1, \dots, \lambda_p)$ only up to the order of the arguments. That is, for $\sigma \in S_p$, we have
 \begin{align*}
 \{\alpha_{\sigma(1)};\dots;\alpha_{\sigma(p)}\}_{(\lambda_{\sigma(1)},\dots,\lambda_{\sigma(p)})}=\{\alpha_1;\dots; \alpha_p\}_\lambda.
 \end{align*}
 However, we will not restrict ourselves to proper partitions for reasons that will later become clear.
\end{Remark}

\begin{Example}
For $\lambda=(2,2)$, we have
\begin{align*}
 \{ \alpha_1; \alpha_2\}_\lambda ={}& [\alpha_1|\alpha_1|\alpha_2|\alpha_2] + [\alpha_1|\alpha_2|\alpha_1|\alpha_2] + [\alpha_1|\alpha_2|\alpha_2|\alpha_1] \\
 &{}+ [\alpha_2|\alpha_1|\alpha_1|\alpha_2] + [\alpha_2|\alpha_1|\alpha_2|\alpha_1] + [\alpha_2|\alpha_2|\alpha_1|\alpha_1].
\end{align*}
\end{Example}

\begin{Lemma}\label{eq: symmetric chains are cycles}
Let $\lambda$ be a $p$-composition of $d \in \NN$ and $\{\alpha_1,\dots,\alpha_p\} \subseteq \Pi$.
Then, the chain $\{\alpha_1;\dots ;\alpha_p\}_\lambda \in A_{2d-1}^1(\Pi)$ is a cycle, i.e.,
\begin{gather*}
 \partial \{a_1;\dots ;a_p\}_\lambda =0.
\end{gather*}
\end{Lemma}

\begin{proof}
 Applying Definition \ref{def: boundary operator of abelian chains}, we obtain the general formula{\samepage
 \begin{align*}
 \partial [\beta_1|\dots| \beta_d] &= \sum_{i=1}^{d-1}\,(-1)^{a_i}[\beta_1|\dots|\beta_i \ast_0 \beta_{i+1} | \dots | \beta_d] \\
 &= \sum_{i=1}^{d-1}\big([\beta_1|\dots|\beta_i, \beta_{i+1} | \dots | \beta_d] - [\beta_1|\dots|\beta_{i+1}, \beta_i | \dots | \beta_d] \big).
 \end{align*}
 Note that $a_i=2i$ is always even in the considered case.}

 Looking at this formula, we notice that the $i$th summand of $\partial [\beta_1|\dots| \beta_d]$ vanishes iff $\beta_i=\beta_{i+1}$. If $\beta_i \neq \beta_{i+1}$, then $[\beta_1|\dots|\beta_{i+1}|\beta_i| \dots | \beta_d]$ is the unique summand of $\{\alpha_1;\dots ;\alpha_p\}_\lambda$ such that the $i$th summand of its boundary cancels out the $i$th summand of $\partial [\beta_1|\dots| \beta_d]$. Since this is a~one-to-one correspondence between non-vanishing summands of $\partial \{\alpha_1;\dots ;\alpha_p\}_\lambda$, the statement follows.
\end{proof}

In the following, the elements $\{\alpha_1;\dots ;\alpha_p\}_\lambda \in Z_{2d-1}^1(\Pi)$ will be referred to as {\it symmetrized cycles}.

We will now explore how symmetrized cycles are related to each other. We start by proving identities for small $m \in \NN$.

\begin{Lemma}
 For $\alpha,\beta,\gamma \in \Pi$ we have the following identities in $H_{3}^1(\Pi)$:
 \begin{align*}
 &\big[\{\alpha\beta\}_{(2)} \big]=\big[\{\alpha\}_{(2)}\big]+\big[\{\beta\}_{(2)}\big]+\big[\{\alpha;\beta\}_{(1,1)}\big],\\
& \big[\{\alpha\beta;\gamma\}_{(1,1)}\big]=\big[\{\alpha;\gamma\}_{(1,1)}\big]+\big[\{\beta;\gamma\}_{(1,1)}\big],\\
 &\big[\{\alpha\}_{(2)}\big]=\big[\big\{\alpha^{-1}\big\}_{(2)}\big].
 \end{align*}
\end{Lemma}

\begin{proof}
 Using the definition of the boundary operator, one can easily verify that the element
 \begin{align*}
 -\{\alpha\beta\}_{(2)}+\{\alpha\}_{(2)}+\{\beta\}_{(2)}+\{\alpha;\beta\}_{(1,1)}
 \end{align*}
 is the boundary of the following element in $A_{4}^1(\Pi)$:
 \begin{align*}
 [\alpha,\beta|\alpha\beta]+[\alpha|\beta,\alpha]+[\beta|\alpha,\beta]-[\alpha,\beta,\alpha,\beta].
 \end{align*}
 Moreover, we have
 \begin{align*}
 \partial\big([\alpha,\beta|\gamma]+[\gamma|\alpha,\beta]\big)= -\{\alpha\beta;\gamma\}_{(1,1)}+ \{\alpha;\gamma\}_{(1,1)}+\{\beta;\gamma\}_{(1,1)}
 \end{align*}
 and
 \begin{align*}
 \partial\big( \big[\alpha|\alpha^{-1},\alpha\big]-\big[\alpha^{-1},\alpha|\alpha^{-1}\big]+\big[\alpha,\alpha^{-1},\alpha,\alpha^{-1}\big] \big)=\{\alpha\}_{(2)}-\big\{\alpha^{-1}\big\}_{(2)}.\tag*{\qed}
 \end{align*}
\renewcommand{\qed}{}\end{proof}

Using the symmetry from Remark \ref{rem: symmetric chains don't depend on order of composition}, the following corollary follows immediately from the previous lemma.

\begin{Corollary}\label{cor: n-form for d=1}
 For $\alpha_1,\alpha_2,\alpha_3 \in \Pi$, we have
 \begin{align*}
 0={}&\big[\{\alpha_1\alpha_2\alpha_3\}_{(2)} \big]- \big[\{\alpha_1\alpha_2\}_{(2)} \big]-\big[\{\alpha_1\alpha_3\}_{(2)} \big]-\big[\{\alpha_2\alpha_3\}_{(2)} \big]\\
 &{}+\big[\{\alpha_1\}_{(2)} \big]+\big[\{\alpha_2\}_{(2)} \big]+\big[\{\alpha_3\}_{(2)} \big].
 \end{align*}
\end{Corollary}

\begin{Lemma}
 For $\alpha, \beta, \gamma,\delta \in \Pi$, we have the following identities in $H_{5}^1(\Pi)$:
 \begin{gather*}
 \big[\{\alpha\beta\}_{(3)} \big]=\big[\{\alpha\}_{(3)}\big]+\big[\{\beta\}_{(3)}\big]+\big[\{\alpha;\beta\}_{(2,1)}\big]+\big[\{\beta;\alpha\}_{(2,1)}\big],\\
 \big[\{\alpha\beta;\gamma\}_{(2,1)}\big]=\big[\{\alpha;\gamma\}_{(2,1)}\big]+\big[\{\beta;\gamma\}_{(2,1)}\big]+\big[\{\alpha;\beta;\gamma\}_{(1,1,1)} \big],\\
 \big[\{\alpha;\beta\gamma\}_{(2,1)}\big]=\big[\{\alpha;\beta\}_{(2,1)}\big]+\big[\{\alpha;\gamma\}_{(2,1)}\big], \\
 \big[\{\alpha\beta;\gamma;\delta\}_{(1,1,1)}\big]= \big[\{\alpha;\gamma;\delta\}_{(1,1,1)}\big]+\big[\{\beta;\gamma;\delta\}_{(1,1,1)}\big]\\
 \big[\{\alpha\}_{(3)}\big]=-\big[\big\{\alpha^{-1}\big\}_{(3)}\big]\\
 \big[\{\alpha;\beta\}_{(2,1)}\big]=\big[\big\{\alpha^{-1};\beta\big\}_{(2,1)}\big].
 \end{gather*}
\end{Lemma}

\begin{proof}
 A long and tedious calculation shows that
 \begin{align*}
 -\{\alpha\beta\}_{(3)}+\{\alpha\}_{(3)}+\{\beta\}_{(3)}+\{\alpha;\beta\}_{(2,1)}+\{\beta;\alpha\}_{(2,1)}
 \end{align*}
 is the boundary of the following element in $A_{6}^1(\Pi)$:
 \begin{gather*}
 [\alpha,\beta|\alpha\beta|\alpha \beta]+[\beta|\alpha,\beta|\alpha \beta]+[\alpha|\beta,\alpha|\beta\alpha]+[\beta|\alpha|\alpha,\beta]+[\alpha|\beta|\alpha,\beta]+[\alpha|\alpha|\alpha,\beta]+[\beta|\beta|\alpha,\beta]\\
 \qquad{}-[\alpha,\beta,\alpha,\beta|\alpha\beta]-[\beta|\alpha,\beta,\alpha,\beta]-[\alpha|\beta,\alpha,\beta,\alpha]
 +[\alpha,\beta,\alpha,\beta,\alpha,\beta].
 \end{gather*}
 Another calculation shows that
 \begin{align*}
 -\{\alpha\beta;\gamma\}_{(2,1)}+\{\alpha;\gamma\}_{(2,1)}+\{\beta;\gamma\}_{(2,1)}+\{\alpha;\beta;\gamma\}_{(1,1,1)}
 \end{align*}
 is the boundary of the following element in $A_{6}^1(\Pi)$:
 \begin{gather*}
 [\alpha,\beta|\alpha\beta|\gamma]+[\alpha,\beta|\gamma|\alpha\beta]+[\gamma|\alpha,\beta|\alpha\beta]+[\gamma|\beta|\alpha,\beta]
 +[\beta|\gamma|\alpha,\beta]+[\gamma|\alpha|\beta,\alpha]+[\alpha|\gamma|\beta,
 \alpha]\\
 \qquad{}+[\alpha|\beta,\alpha|\gamma]+[\beta|\alpha,\beta|\gamma]-[\alpha,\beta,\alpha,\beta|\gamma]-[\gamma|\alpha,\beta,\alpha,\beta].
 \end{gather*}
 Moreover, we have
 \begin{align*}
 \partial \big([\alpha|\alpha|\beta,\gamma]+[\alpha|\beta,\gamma|\alpha]+[\beta,\gamma|\alpha|\alpha] \big)= -\{\alpha;\beta\gamma\}_{(2,1)}+\{\alpha;\beta\}_{(2,1)}+\{\alpha;\gamma\}_{(2,1)}.
 \end{align*}
 The fourth identity follows from the element
 \begin{align*}
 -\{\alpha\beta;\gamma;\delta\}_{(1,1,1)}+\{\alpha;\gamma;\delta\}_{(1,1,1)}+\{\beta;\gamma;\delta\}_{(1,1,1)}
 \end{align*}
 being the boundary of the following element in $A_{6}^1(\Pi)$:
 \begin{gather*}
 [\alpha,\beta|\gamma|\delta]+[\gamma|\alpha,\beta|\delta]+[\gamma|\delta|\alpha,\beta]+[\alpha,\beta|\delta|\gamma]+[\delta|\alpha,\beta|\gamma]+[\delta|\gamma|\alpha,\beta].
 \end{gather*}
 Furthermore, the element
 \begin{align*}
 \{\alpha\}_{(3)}+ \big\{\alpha^{-1}\big\}_{(3)}
 \end{align*}
 is the boundary of the following element in $A_{6}^1(\Pi)$:
 \begin{gather*}
 \big[\alpha^{-1}|\alpha^{-1}|\alpha^{-1},\alpha\big]-\big[\alpha^{-1}|\alpha^{-1},\alpha|\alpha\big]+\big[\alpha^{-1},\alpha|\alpha|\alpha\big]-\big[\alpha^{-1}|\alpha,\alpha^{-1},\alpha,\alpha^{-1}\big]\\ \qquad{}+\big[\alpha,\alpha^{-1},\alpha,\alpha^{-1}|\alpha\big] +\big[\alpha,\alpha^{-1},\alpha,\alpha^{-1},\alpha,\alpha^{-1}\big].
 \end{gather*}
 Finally, the element
 \begin{align*}
 \{\alpha;\beta\}_{(2,1)}-\big\{\alpha^{-1};\beta\big\}_{(2,1)}
 \end{align*}
 is the boundary of the following element in $A_{6}^1(\Pi)$:
 \begin{gather*}
 \big[\alpha^{-1}|\alpha^{-1},\alpha|\beta\big]-\big[\alpha^{-1},\alpha|\alpha|\beta\big]+\big[\alpha^{-1}|\beta|\alpha^{-1},\alpha\big]-\big[\alpha^{-1},\alpha|\beta|\alpha\big]\\
 \qquad{}+\big[\beta|\alpha^{-1}|\alpha^{-1},\alpha\big]-\big[\beta|\alpha^{-1},\alpha|\alpha\big]-\big[\beta|\alpha,\alpha^{-1},\alpha,\alpha^{-1}\big]-\big[\alpha,\alpha^{-1},\alpha,\alpha^{-1}|\beta\big].\tag*{\qed}
 \end{gather*}
\renewcommand{\qed}{}\end{proof}

Again, using the symmetry from Remark \ref{rem: symmetric chains don't depend on order of composition}, the following corollary follows immediately from the previous lemma.

\begin{Corollary} \label{cor: n-form for d=2}
 For $\alpha_1,\alpha_2,\alpha_3,\alpha_4 \in \Pi$, we have
 \begin{gather*}
 0=\big[\{\alpha_1\alpha_2\alpha_3\alpha_4\}_{(3)} \big]-\big[\{\alpha_1\alpha_2\alpha_3\}_{(3)} \big]
 -\big[\{\alpha_1\alpha_2\alpha_4\}_{(3)} \big]
-\big[\{\alpha_1\alpha_3\alpha_4\}_{(3)} \big]\nonumber\\
 \phantom{0=}{}-\big[\{\alpha_2\alpha_3\alpha_4\}_{(3)} \big]
 +\big[\{\alpha_1\alpha_2\}_{(3)} \big]
+\big[\{\alpha_1\alpha_3\}_{(3)} \big]
 +\big[\{\alpha_1\alpha_4\}_{(3)} \big]
+\big[\{\alpha_2\alpha_3\}_{(3)} \big]\nonumber\\
\phantom{0=}{}+\big[\{\alpha_2\alpha_4\}_{(3)} \big]
 +\big[\{\alpha_3\alpha_4\}_{(3)} \big]
 -\big[\{\alpha_1\}_{(3)} \big]
-\big[\{\alpha_2\}_{(3)} \big]
 -\big[\{\alpha_3\}_{(3)} \big]
 -\big[\{\alpha_4\}_{(3)} \big],
 \\
 0=\big[\{\alpha_1\alpha_2\alpha_3;\alpha_4\}_{(2,1)}\big]-\big[\{\alpha_1\alpha_2;\alpha_4\}_{(2,1)}\big]-\big[\{\alpha_1\alpha_3;\alpha_4\}_{(2,1)}\big]-\big[\{\alpha_2\alpha_3;\alpha_4\}_{(2,1)}\big]\nonumber\\
 \phantom{0=}{}+\big[\{\alpha_1;\alpha_4\}_{(2,1)}\big]+\big[\{\alpha_2;\alpha_4\}_{(2,1)}\big]+\big[\{\alpha_3;\alpha_4\}_{(2,1)}\big].
 \end{gather*}
\end{Corollary}

Both Corollaries \ref{cor: n-form for d=1} and~\ref{cor: n-form for d=2} give evidence for the following conjecture:

\begin{conje} \label{conje: symmetrized cycles are n-forms in arguments}
 Let $\lambda=(\lambda_1,\dots,\lambda_p)$ be a $p$-composition of $d \in \NN$ and
 \[\{\alpha_1,\dots,\alpha_{i-1},\beta_1, \dots, \beta_{\lambda_{i}+1},\alpha_{i+1}, \dots ,\alpha_p\} \subseteq \Pi.\]
 Then we have the following identity in $H_{2d-1}^1(\Pi)$:
 \begin{align*}
 0={}&\big[\{\alpha_1;\dots;\alpha_{i-1};\beta_1\dots\beta_{\lambda_{i}+1};\alpha_{i+1}; \dots ;\alpha_p\}_\lambda \big] \\
 &{}-\sum_{1 \leq j \leq \lambda_{i}+1} \, \big[\{\alpha_1;\dots;\alpha_{i-1};\beta_1\dots \hat{\beta}_j\dots \beta_{\lambda_{i}+1};\alpha_{i+1}; \dots ;\alpha_p\}_\lambda \big] \\
 &{}+\sum_{1\leq k < l \leq \lambda_{i}+1} \, \big[\{\alpha_1;\dots;\alpha_{i-1};\beta_1\dots \hat{\beta}_k \dots \hat{\beta}_l\dots \beta_{\lambda_{i}+1};\alpha_{i+1}; \dots ;\alpha_p\}_\lambda \big]\\
 &\qquad\qquad\vdots\\
 &{}+(-1)^{\lambda_i}\sum_{1\leq m \leq \lambda_{i}+1}\, \big[\{\alpha_1;\dots;\alpha_{i-1};\hat{\beta}_1\dots \beta_m\dots \hat{\beta}_{\lambda_{i}+1};\alpha_{i+1}; \dots ;\alpha_p\}_\lambda \big].
 \end{align*}
 Moreover, we have
 \begin{align*}
 \big[\{\alpha_1;\dots;\alpha_{i-1};\beta;\alpha_{i+1}; \dots;\alpha_p\}_\lambda \big]=(-1)^{\lambda_i}\big[\big\{\alpha_1;\dots;\alpha_{i-1};\beta^{-1};\alpha_{i+1}; \dots;\alpha_p\big\}_\lambda \big].
 \end{align*}
\end{conje}

\begin{Remark}\label{rem: n-form}
 Let $A$ be an abelian group and let $\lambda=(\lambda_1,\dots,\lambda_p)$ be a $p$-composition of $m \in \NN$. Moreover, let $\theta \in A^{2d-1}_1(\Pi,A)$ be a cochain. We can define a map $\theta_\lambda\colon\Pi^{\times p} \to A$ by setting
 \begin{align*}
 \theta_\lambda(\alpha_1;\dots;\alpha_p):=\theta\big( \{\alpha_1;\dots;\alpha_p\}_\lambda\big).
 \end{align*}
 Now, we can reinterpret Conjecture \ref{conje: symmetrized cycles are n-forms in arguments} as follows: If $\theta$ is a cocycle, then $\theta_\lambda$ is a $\lambda_i$-form in the $i$th argument. Moreover, from Lemma \ref{eq: symmetric chains are cycles} we know that in this case $\theta_\lambda$ only depends on the cohomology class of $\theta \in Z^{2d-1}_1(\Pi,A)$.
\end{Remark}

\begin{Example} \label{ex: Rosso condition depends on Cohclass}
 Let $\KK$ be a field and $d,n \in \NN$. We set $\Pi=\ZZ^n$ and $A=\KK^\times$. Let $(\alpha_1,\dots, \alpha_n)$ be the standard basis of $\ZZ^n$. We are now in the setting of Definition \ref{def: defn of qk's}.

 Let $\theta \in Z^{2d-1}_1(\ZZ^n,\KK^\times)$ be a $(2d-1)$-cocycle. For $i_1, \dots, i_d \in \{1,\dots,n\}$, we define
 \begin{align*}
 q_{i_1,\dots,i_d}:= \theta([\alpha_{i_1}|\dots |\alpha_{i_d}]) \in \KK^\times.
 \end{align*}
 Then, for $\ell,j \in \{1, \dots,n\}$ and $k \in \{0,\dots,d\}$, we have
 \begin{align*}
 q_k = \theta_{(d-k,k)}(\alpha_\ell,\alpha_j).
 \end{align*}
 In particular, by the previous remark the Rosso condition $\tilde{R}_m=0$ only depends on the cohomology class $[\theta] \in H^{2d-1}_1\big(\ZZ^n,\KK^\times\big)$.
\end{Example}

\pdfbookmark[1]{References}{ref}
\LastPageEnding

\end{document}